\documentclass[12pt,a4paper,reqno]{amsart}

\usepackage{amssymb,amsfonts,amsthm,amsmath,amsopn,amstext,amscd,latexsym}
\usepackage{bbm}
\usepackage{enumerate}
\usepackage{setspace}
\usepackage[all]{xy}
\usepackage{geometry}
\usepackage{float, wrapfig, wasysym}
\usepackage{color}

\usepackage[pdftex]{graphicx}
\usepackage{mdwlist}

\usepackage{etoolbox}
\appto\appendix{\addtocontents{toc}{\protect\setcounter{tocdepth}{0}}}

\usepackage[pdftex,pdfpagelabels,bookmarks=false,hyperindex,hyperfigures]{hyperref}
\hypersetup{colorlinks=true}

\newcommand{\K}{\mathbb K}
\newcommand{\bbL}{{\mathbb L}}

\newcommand{\bbP}{{\mathbb P}}

\newcommand{\cB}{{\mathcal B}}

\newcommand{\cF}{{\mathcal F}}

\newcommand{\cR}{{\mathcal R}}
\newcommand{\cS}{{\mathcal S}}

\newcommand{\N}{\mathbb N}
\newcommand{\Z}{\mathbb Z}
\newcommand{\Q}{\mathbb Q}
\newcommand{\R}{\mathbb R}
\newcommand{\C}{\mathbb C}

\newcommand{\ra}{\rightarrow}
\newcommand{\lra}{\longrightarrow}

\renewcommand{\1}{\mathbbm{1}}
\newcommand{\inv}{^{-1}}

\DeclareMathOperator{\Exp}{Exp}

\DeclareMathOperator{\id}{id}
\DeclareMathOperator{\Aut}{Aut}
\DeclareMathOperator{\Ext}{Ext}
\DeclareMathOperator{\Int}{Int}
\DeclareMathOperator{\Frac}{Frac}

\DeclareMathOperator{\im}{im}

\DeclareMathOperator{\lex}{lex}

\DeclareMathOperator{\Hom}{Hom}
\newcommand{\vAut}{v\text{-}\!\Aut}
\newcommand{\oAut}{o\text{-}\!\Aut}
\newcommand{\uAut}{1\text{-}\!\Aut}
\newcommand{\pExt}{\Psi\text{-}\!\Ext}
\DeclareMathOperator{\Str}{+}
\newcommand{\gexp}[1]{#1\text{-}\!\Exp}

\DeclareMathOperator{\Cr}{Cr}
\DeclareMathOperator{\PGL}{PGL}

\newcommand{\pow}[1]{\!\left(\!\left( #1 \right)\!\right)}

\newcommand{\setbr}[1]{\left\{ #1 \right\}}
\newcommand{\brackets}[1]{\left( #1 \right)}

\DeclareMathOperator*\supp{supp}
\DeclareMathOperator*\Supp{Supp}
\newcommand{\molt}{\times}
\DeclareMathOperator{\Char}{char}

\newcommand{\func}[5]{
	\begin{center}
		\begin{tabular}{rccc}
			$#1\colon$ &  $#2$ &$\longrightarrow$ & $#3$\\
			& $#4$ & $\longmapsto$ & $#5$
		\end{tabular}
	\end{center}
}

\renewcommand{\epsilon}{\varepsilon}
\renewcommand{\phi}{\varphi}
\renewcommand{\theta}{\vartheta}

\theoremstyle{definition}

\newtheorem{defn}{Definition}[subsection]

\newtheorem{es}[defn]{Example}

\newtheorem{ess}[defn]{Examples}

\newtheorem{notation}[defn]{Notation}

\theoremstyle{plain}

\newtheorem{lemma}[defn]{Lemma}

\newtheorem*{lemma*}{Lemma}

\newtheorem{prop}[defn]{Proposition}

\newtheorem*{prop*}{Proposition}
\newtheorem{theorem}[defn]{Theorem}

\newtheorem*{teorema*}{Theorem}

\newtheorem{cor}[defn]{Corollary}

\newtheorem{rmk}[defn]{Remark}

\newtheorem*{cor*}{Corollary}

{\left\lbrace\begin{array}{@{}l@{}}}%
	{\end{array}\right.}

\newcommand{\vertiii}[1]{{\left\vert\kern-0.25ex\left\vert\kern-0.25ex\left\vert #1 
		\right\vert\kern-0.25ex\right\vert\kern-0.25ex\right\vert}}

\newcounter{nootje}
\usepackage[autostyle]{csquotes}
\usepackage[backend=biber, style=alphabetic]{biblatex}
\numberwithin{equation}{section}
\begin{document}

\begin{abstract}
	Let $ k $ be a field, $ G $ a totally ordered abelian group and $ \K = k\pow{G} $ the maximal field of generalised power series, endowed with the canonical valuation $ v $ \cite{hahn}.
	We study the group $ \vAut K $ of valuation preserving automorphisms of a subfield $ k(G)\subseteq K\subseteq \K $, where $ k(G) $ is the fraction field of the group ring $ k[G] $.
	Under the assumption that $ K $ satisfies two\textit{ lifting properties} we are able to generalise and refine the decomposition of $ \vAut\K $ \cite{hofbergerthesis} and prove a structure theorem decomposing $ \vAut K $ into a 4-factor semi-direct product of notable subgroups.
	We then identify a large class of Hahn fields satisfying the two aforementioned lifting properties.
	Next we focus on the group of \textit{strongly additive} automorphisms of $ K $.
	We give an explicit description of the group of strongly additive \emph{internal} automorphisms in terms of the groups of homomorphisms $ \Hom(G,k^\times) $  of $ G $ into $ k^\times $ and $ \Hom(G,1+I_K) $ of $ G $ into the group of $ 1 $-units of the valuation ring of $ K $.
	Finally, we specialise our results to some relevant special cases. In particular, we extend the work of Schilling \cite{schilling44} on the field of Laurent series and that of \cite{deschamps:puiseux} on the field of Puiseux series. 
\end{abstract}

\title[Automorphisms of fields of generalised power series]{The automorphism group of a valued field \\of generalised formal power series}

\author[S.~Kuhlmann]{Salma Kuhlmann}
\author[M.~Serra]{Michele Serra}

\address{Fachbereich Mathematik und Statistik, Universität Konstanz, 78457 Konstanz, Germany}
\email{salma.kuhlmann@uni-konstanz.de}

\address{Fachbereich Mathematik und Statistik, Universität Konstanz, 78457 Konstanz, Germany}
\email{michele.serra@uni-konstanz.de}

\date{\today \\ \textup{2020} \textit{Mathematics Subject Classification}: 13J05 (12J20 16W60 06F20)}


\maketitle


	\section{Introduction}
	In his paper \cite{schilling44}, Schilling studied the $ k $-automorphisms of the field $ \bbL=k\pow{\Z} $ of formal Laurent series over a coefficient field $ k $. First, he observed \cite[Lemma~1]{schilling44} that all $ k $-automorphisms are necessarily valuation preserving.
	Moreover, since the ordered abelian group $ \Z $ admits only  the trivial automorphism, all $ k $-automorphisms of $ \bbL $ are 
	\emph{internal} (Definition~\ref{def-int-field}).
	
	Schilling then proves that the group of $ k $-automorphisms is isomorphic to the group of units of the valuation ring of $ \bbL $ endowed with a particular group operation (see Section~\ref{subsect-laurent-series} for more details). 
	
	In this paper, inspired by Schilling's ideas, we study the group $ \vAut K$ of \emph{valuation preserving automorphisms} of a Hahn field $ K $, and its subgroups $\vAut_{(k)} K,$ $\vAut_k K $ of $ k $\emph{-stable} resp.\ $ k $\emph{-automorphisms} (Notation~\ref{not:Aut-groups}\eqref{not:sub-v-Aut}). \emph{Hahn fields} are distinguished subfields of fields of generalised power series $ \K=k\pow{G} $ for a field $ k $ and a totally ordered abelian group $G$ (Definition~\ref{def:hahn-field}).
	The group $ \Int\Aut K $ of internal automorphisms is an important normal subgroup of $ \vAut K $.
	In order to further analyse $ \vAut K $, we impose an additional assumption on $ K $.
	A Hahn field $ K $ such that every pair $ (\rho,\tau)\in\Aut k \times \oAut G $ lifts to an automorphism of $ K $ is said to have the \emph{first lifting property} (Notation~\ref{not:Aut-groups} and Definition~\ref{def-lifting-fields}). This allows us to identify another important subgroup, namely the group $ \Ext\Aut K $ of \emph{external} valuation preserving automorphisms (Definition~\ref{def-ext-field}).
	
	A result of Hofberger {\cite[Satz 2.2]{hofbergerthesis}} allows to present the group $ \vAut\K $ of the \emph{maximal} Hahn field $ \K $ as the semi-direct product $ \Int\Aut\K \rtimes\Ext\Aut\K $. 
	In Theorem~\ref{hofberger-semi-direct} we generalise Hofberger's result and show that all Hahn fields which satisfy the first lifting property admit Hofberger's decomposition.
	
	We further refine Hofberger's result by first noticing that the group $ \Ext\Aut K $ is isomorphic to the direct product $ \Aut k\times \oAut G $. A detailed study of the factor $ \oAut G $ is provided in \cite{droste-goebel} and generalised in \cite{ours-groups}.
	
	Secondly, we decompose the group $ \Int\Aut K $ (Theorem~\ref{hofberger-enhanced}) of a Hahn field $ K $ which satisfies the \emph{second lifting property} (Definition~\ref{def:G-exp-closed}) by identifying two further subgroups of fundamental importance: the group $ \uAut K $
	of \emph{1-automorphisms} and the group $ \gexp{G}K $ of $ G $\emph{-exponentiations} (Definitions~\ref{def:1-aut} and \ref{def:g-exp}).
	In Corollary~\ref{cor:1-Aut-determines-vAut} we deduce that $ \vAut K $
	is completely determined, up to isomorphism, by $ \uAut K $. 
	Similar results are obtained for $ \vAut_{k}K $ and $ \vAut_{(k)}K $.
	
	An interesting large class of Hahn fields is that of \emph{Rayner fields} (Definition~\ref{def:rayner}). 
	We characterise Rayner fields which fulfil the \emph{canonical first lifting property} (Definition~\ref{def:CLP}) and prove that all Rayner fields satisfy the second lifting property.
	We follow up our study on Rayner fields in \cite{KKS-rayner-structures} and \cite{KKS-linear-recurrence}.
	
	Pursuing our generalisation of Schilling's results, we study the group $ \vAut^{\Str} K$ of \emph{strongly additive} (Definition~\ref{def-strongly-linear}) valuation preserving automorphisms (and its subgroups $\vAut_{(k)}^{\Str} K$,  $ \vAut_k^{\Str} K $) of a Hahn field $ K $.	
	In particular, in Theorem~\ref{thm:main} we describe $ \vAut_{(k)}^{\Str}K $, for a Hahn field $ K $ which satisfies the first and second lifting property, in terms of the valuation invariants of $ K $.
	Finally, to illustrate our results, 
	we provide a detailed study of two special cases. 
	If the group $ G $ is finitely generated we obtain a generalisation of a result of \cite{schilling44} on the field
	$ \bbL $ of Laurent series.
	If $ G $ is divisible and finite dimensional, we obtain increasingly precise results depending on our assumptions on $ k $.
	For the field $ \bbP $ of Puiseux series we find descriptions of $ \vAut_{(k)}\bbP $ and $ \vAut_{k}\bbP $ that generalise results of \cite{deschamps:puiseux}.
	
	\smallskip	
	The paper is organised as follows.
	In Section~\ref{sec:definitions} we introduce the fundamental definitions and establish some notation that will be used throughout the paper. In this section Example~\ref{es:counter-non-arch} is of particular interest: we show that if the coefficient field $ k $ is not archimedean, order preserving automorphisms need not be valuation preserving.
	
	In Section~\ref{section:decomposition} we introduce two notions of lifting property, which allow us to obtain several decomposition theorems.
	
	\noindent
	Subsection~\ref{sec:lifting-property} is devoted to the study of the first lifting property (Definition~\ref{def-lifting-fields}). For a Hahn field $ K $ satisfying the first lifting property, we introduce two important subgroups of $ \vAut K $: the normal subgroup $ \Int\Aut K $ and its complement $ \Ext\Aut K $ (Definitions~\ref{def-int-field} and \ref{def-ext-field}).
	
	\noindent
	The main result of Subsection~\ref{subsec:Hofberger} is Theorem~\ref{hofberger-semi-direct}, which generalises Hofberger's decomposition to any Hahn field with the first lifting property. This will be the starting point for a much deeper investigation of the structure of the automorphism group.
	
	\noindent
	In Subsection~\ref{subsec:internal} we introduce the canonical second lifting property.
	For a Hahn field $ K $ satisfying the canonical second lifting property, we further introduce two important subgroups of $ \Int\Aut K $: the normal subgroup $ \uAut K $ of $ 1 $-auto\-mor\-phisms and its complement, the group $ \gexp{G}K $ of $ G $-exponentiations (Definition~\ref{def:G-exp}). 
	
	\noindent
	In Subsection~\ref{subsec:intern-decomp} we provide a decomposition of $ \Int\Aut K $ into a semi-direct product of the subgroups $ \uAut K $ and $ \gexp{G}K $, for a Hahn field $ K $ satisfying the  canonical second lifting property.
	
	\noindent
	Subsection~\ref{sec:CLP} is dedicated to a special version of the first lifting property: the canonical first lifting property.
	This allows to simplify the decomposition of the group $ \vAut_{(k)}K $ of $ k $-stable automorphism, as will be shown in Subsection~\ref{sec:decomp-general}.
	
	\noindent
	In Subsection~\ref{subsec:rayner} we present the class of Rayner fields, first introduced in \cite{rayner1968}. All these fields satisfy the canonical second lifting property and include, for example, the $ \kappa $\emph{-bounded} Hahn fields (Example~\ref{rayner-examples-ours}) introduced in \cite{alling:existence-eta-alpha} and, in particular, the maximal Hahn fields. The main result of this section is a characterisation of Rayner fields with the canonical first lifting property (Proposition~\ref{prop:rayner-criterion}).
	
	\noindent
	In Subsection~\ref{sec:decomp-general}, combining results from the previous sections, we obtain two decompositions into a 4-factor semi-direct product. One for the groups $ \vAut K $ and $ \vAut_{k}K $ for a Hahn field $ K $ with the first and canonical second lifting property (Theorem~\ref{hofberger-enhanced}) and one for $ \vAut_{(k)}K $ under the extra assumption, that $ K $ has the canonical first lifting property (Proposition~\ref{hofberger-enhanced-k-stable}).

	Section~\ref{sec:strongly-linear} focuses on strong additivity. 
	
	\noindent
	In Subsection~\ref{subsec:structure} we show that Hofberger's decomposition, in its generalised and refined form, also holds if we restrict to the group $ \vAut^{\Str}K $  (Proposition~\ref{prop:hofberger-str-linear} and Proposition~\ref{rmk:1-Aut-is-normal}). This way we obtain a detailed description of $ \vAut^{\Str} K$ and its subgroups $\vAut_{(k)}^{\Str} K,\ \vAut_k^{\Str}K $ (Theorem~\ref{thm:hofberger-enhanced-str}), which is the main result of this subsection.
	
	\noindent
	Subsection~\ref{subsec:general} is devoted to a deeper investigation of the component $ \uAut^{\Str}_{k} K $ appearing in Theorem~\ref{thm:hofberger-enhanced-str}.
	We do this in terms of the \emph{$ K $-summable homomorphisms} $ \Hom^{\Str}(G,1+I_K) $ (Definition~\ref{def:summable-map}) from the value group into the group of 1-units of the valuation ring. This is the main result of the section (Theorem~\ref{thm:main}), which provides a decomposition of $ \vAut_{(k)}^{\Str}K $ and $ \vAut_{k}^{\Str}K $ purely in terms of the valuation invariants of $ K $.
	
	Section~\ref{sec:examples} is devoted to the explicit description of the automorphism groups in some special cases.
	
	\noindent
	In Subsection~\ref{fin-gen-G:subsec} we study the case of a Hahn field $ K\subseteq k\pow{G} $ with a finitely generated exponent group, necessarily of the form $ G=\Z^n $. In this case, if $ K $ has the canonical second lifting property, we can explicitly describe $ \gexp{G}K $ in terms of $ k $ and the number $ n $ of generators of $ G $ (Theorem~\ref{general-Z^n:thm}). If $ G=\Z^n $ is ordered lexicographically, we can moreover represent $ \oAut G $ as a group of matrices. Applying this, we give a description of the groups $ \vAut_{(k)}^{\Str}K$ and $ \vAut_k^{\Str}K $, which depends solely on $ k $, $ 1+I_K $ and $ n $ (Theorem~\ref{srt-add-Z^n:thm}).
	
	\noindent
	In Subsection~\ref{subsect-laurent-series} we apply the results from Subsection~\ref{fin-gen-G:subsec} to the field $ \bbL $ of Laurent series.
	We notice that $ \vAut\bbL=\vAut^{\Str}\bbL $ and obtain a precise description of $ \vAut_{(k)}\bbL $ in terms of $ 1+I_\bbL $, $ k^\times $ and $ \Aut k $ (Theorem~\ref{thm:laurent-main}).
	Schilling's result on the $ k $-automorphisms of $ \bbL $ can be derived as a special case (Corollary~\ref{schilling:cor}).
	We also provide a sharpening of Theorem~\ref{thm:laurent-main} for the case of an ordered coefficient field $ k $ to characterise the group $ \oAut\bbL $ of order preserving automorphisms.
	
	\noindent
	Subsection~\ref{cremona-1:es} is also dedicated to the case $ G=\Z $, but focusses on the field $ k(\Z) $. The group $ \Cr_1(k):=\Aut_k k(\Z) $ is the \emph{Cremona\footnote{Cremona groups are of fundamental importance in algebraic geometry (see \cite{deserti}).} group in dimension $ 1 $ over $ k $}.
	We identify the groups $ \vAut_k k(\Z) $ and $ \uAut_k k(\Z) $ as subgroups of $ \Cr_1(k) $.
	
	\noindent
	In Subsection~\ref{divisible:subsec} we study the case of a Hahn field $ K\subseteq k\pow{G} $ where $ G $ is a totally ordered, divisible, abelian group, of finite dimension $ d $ over $ \Q $ and $ k $ a real closed field. 
	In analogy to Subsection~\ref{fin-gen-G:subsec}, we provide a description of $ \oAut G $ as a group of rational matrices and, under the extra assumption that $ K $ is henselian, we give an explicit description of the groups $ \vAut_{(k)}^{\Str}K$ and $ \vAut_k^{\Str}K $, depending only on $ k $, $ 1+I_K $ and $ d $ (Theorem~\ref{hensel-char0-str-add:thm}).
	
	\noindent
	In Subsection~\ref{subsect-puiseux} we consider the field $ \bbP $ of Puiseux series. After observing that $ \vAut\bbP = \vAut^{\Str}\bbP $ (Proposition~\ref{prop:puiseux-str-lin}), we apply our general results and retrieve decompositions analogous to those appearing in \cite{deschamps:puiseux}.

\section{Definitions and notations}\label{sec:definitions}
\setcounter{subsection}{1}
Let $ (G,+,0,<) $ be a totally ordered abelian group\footnote{Unless explicitly specified otherwise, all orderings will be total.} and $ k $ a field.
Consider the set $ k^G $ of all functions from $ G $ to $ k $ and for an element $ a\in k^G $ define the \emph{support of} $ a $ to be the set $ \supp a := \{g\in G : a(g)\neq 0 \} $. We denote by $ \K=k\pow{G} $ the set of all elements of $ k^G $ with well ordered support.
Denote a function $ a\in k\pow{G} $ by the formal power series expression $a= \sum_{g\in G}a_gt^g $ where $ a_g :=a(g) $. In particular, the coefficient $ a_0:= a(0) $ will play an important role and we call it the \emph{constant term} of $ a $.
We define the following two operations: for $ a=\sum_G a_gt^g,\ b=\sum_G b_gt^g \in \K $ set
\[
a+b := \sum_{g\in G} (a_g+b_g)t^g\qquad
\text{and}\quad
 ab := \sum_{g\in G}c_gt^g\quad \text{where }
c_g=\sum_{\substack{r,s\in G\\r+s=g}}a_rb_s. \]
It was shown by Hahn \cite{hahn} that these two operations are well defined and make $ \K $ into a field. We call $ \K $ the \emph{maximal Hahn field over $ k $ with exponents in $ G $\label{maximal:def}}.

The \emph{group ring} $ k[G] $ is the unitary subring of $ \K $ consisting of series with finite support (generalised polynomials). It is an integral domain and we denote its field of fractions by $ k(G):=\Frac (k[G])\subseteq \K $.
\begin{defn}\label{def:hahn-field}
	A \emph{Hahn field} is a field $ K $ such that $ k(G)\subseteq K \subseteq \K $. 
	We define a valuation $ v_{\min}^K $ on a Hahn field $ K $ by setting $ v_{\min}^K(a)=\min\supp(a) $ for $ a\neq 0 $ and $ v_{\min}^K(0)=\infty $.
	We will call this the \emph{canonical valuation} on $ K $. Whenever the context is clear we will simply write $ v $ instead of $ v_{\min}^K $.
	For  $a\in K^\times $ we define the \emph{first coefficient of}
$ a $ to be $ a_{v(a)}:= a(v(a)) $.
\end{defn}
\begin{notation}
\label{not:general-KvG}
For a Hahn field $ K $ we will denote by $ v $ its canonical valuation and by $ G = v(K^\times) $ its value group. We denote by $ t^G $ the multiplicative group of monic monomials $ t^G:=\{ t^g:g\in G \} $.
The valuation ring $ R_K $ is $ k\pow{G^{\geq 0}}\cap K $ (elements with non-negative value) and the valuation ideal $ I_K $ is 
$ k\pow{G^{>0}}\cap K $ (elements with positive value). The residue field is $ \bar{K} = R_K/I_K $.

If $ k $ is totally ordered by $ <_k $, we order $ \K $ lexicographically by setting $ a>_{\lex}0\Leftrightarrow a_{v(a)}>_k0 $. Then $ (\K,<_{\lex}) $ is a totally ordered field and so is any subfield $ K\subseteq \K $. When considering the lexicographic order on a Hahn field $ K $ we will omit the subscript and simply denote the ordering by $ < $.
\end{notation}

\begin{rmk}\label{coeff-isom}\label{not:general-KvG-2}
\begin{enumerate}[(i)]
	\item The residue field is isomorphic to $ k $ via the canonical isomorphism $f_c\colon \bar K \to k,\ a+I_K\mapsto a_0 $, for all $ a\in  R_K $.
	Note that, if $ f\colon\bar K \to k $ is any other isomorphism, then there exists a uniquely determined $ \rho_f\in\Aut k $, defined by $ \rho_f(a_0)=f(a_0+I_K) $, such that $ f = \rho_f f_c $.
\item 
The units of the valuation ring are the elements with null value, so the group of units $ R_K^\times $ is the direct sum $ U_K:=I_K + k^\times $. 
\item 
A subgroup of $ U_K $ that will be relevant in the sequel is that of \emph{1-units} (units with constant term 1), i.e., $ 1+I_K $. In fact, $ U_K $ is (isomorphic to) the direct product $ U_K \simeq (1+I_K)\times k^\times $.
\item
In the sequel we will consider the group $ \Hom(G,U_K) $ of homomorphisms of $ (G,+) $ into $ (U_K,\cdot) $, endowed with pointwise multiplication $ \cdot $
(and similarly for $ \Hom(G,1+I_K) $ and $ \Hom(G,k^\times) $).
Since $ U_K = (1+I_K)\times k^\times  $ we have $$  (\Hom(G,U_K),\cdot) \simeq (\Hom(G,1+I_K),\cdot) \times (\Hom(G,k^\times),\cdot).  $$
\end{enumerate}
\qed
\end{rmk}

\begin{defn}\label{def:k,k-stab,v-,o-Aut}
An automorphism $ \tau $ of the group $ (G,+) $ is said to be \emph{order preserving} if, for all $ g\in G $ we have $ g>0\Rightarrow \tau(g)>0 $.
If $ K $ is an ordered Hahn field, an \textit{order preserving automorphism of} $ K $ is defined in a similar way.
An automorphism $ \sigma \in\Aut K $ is \emph{valuation preserving} if there exists a (necessarily unique) order preserving automorphism $ \tau $ of $ G $ such that, for all $ a\in K $, we have $ \tau(v(a)) = v(\sigma(a)) $.
An automorphism $ \sigma \in \Aut K $ is a
$ k $\emph{-stable automorphism} if $ \sigma(k)=k $ and a 
$ k $\emph{-automorphism} if $ \sigma|_k=\id_k $.
Notice that $ \sigma\in\Aut K $ is a $ k $-automorphism if and only if it is $ k $-linear (when we view $ K $ as a $ k $-vector space).
\end{defn}

We recall the definition of a henselian Hahn field. For the corresponding definition for a general valued field see, for example, \cite[Section~4.1]{engler-prestel} 
\begin{defn}\label{henselian:def}
	Let $ K\subseteq k\pow{G} $ be a Hahn field. 
	For a polynomial $ p=p_0+p_1X+\ldots+p_nX^n\in R_K[X] $ let $ \bar p := (p_0)_0 + (p_1)_0 X +\ldots+ (p_n)_0X^n \in k[X]$.
	Then $ K $ is \emph{henselian} if for every polynomial $ p\in R_K[X] $ and every $ a\in R_K $ such that $ \bar p(a_0)=0 $ and $ \bar p'(a_0)\neq 0  $ there exists $ b\in R_K $ such that $ b_0 = a_0 $ and $ p(b)=0 $.
\end{defn}

\begin{notation}\label{not:Aut-groups}
	\begin{enumerate}[(i)]
		\item \label{not:sub-k-Aut}
		The $ k $-stable automorphisms form a group under composition that we will denote by $ \Aut_{(k)}K $.
		The $ k $-automorphisms form a group under composition that we will denote by $ \Aut_kK $.
		\item \label{not:sub-v-Aut}
		The valuation preserving automorphisms of $ K $ form a group under composition that we will denote by $ \vAut K $.
		We will also write $ \vAut_{(k)} K = \vAut K \cap \Aut_{(k)}K $ and $ \vAut_kK = \vAut K \cap \Aut_kK $.
		\item \label{not:sub-o-Aut}
		The order preserving automorphisms of $ G $ form a group under composition that we will denote by $ \oAut G $.
		If $ K $ is an ordered Hahn field, $ \oAut K $ is defined in a similar way.
		
	\end{enumerate}
\end{notation}

\begin{rmk}[Characterisations of valuation preserving automorphisms]
	\label{rmk:val-pres-induces-ord}\label{rmk:ord-implies-val-pres}\label{not:sub-induced-aut}
		Various characterisations of valuation preserving automorphisms in terms of the valuation ring, valuation ideal and group of units are given in \cite[Theorem~4.2]{kuhlmann-matusinski-point}.
	In particular, for all $ \sigma \in \Aut K $:	
		\[ \sigma \in \vAut K \Leftrightarrow
		\sigma(R_K) = R_K \Leftrightarrow
		\sigma(I_K) = I_K \Leftrightarrow
		\sigma(U_K) = U_K \Leftrightarrow
		\sigma(1+I_K) = 1+I_K	\]
		\begin{enumerate}[(i)]
	\item So, for $ \sigma \in \vAut K $ we have 
		 \[\sigma(k^\times\times(1+I_K)) = k^\times\times(1+I_K) \text{ and } \sigma(1+I_K) = 1+I_K.  \]
		 This however does not imply, in general, that $ \sigma(k^\times) = k^\times $, i.e., does not imply, in general, that $ \sigma\in\vAut_{(k)}K. $
		 
	\item Furthermore $ \sigma\in \vAut K $ if and only if, for all $ a,b\in K $ we have $$  v(a)<v(b)\iff v(\sigma(a))<v(\sigma(b)).  $$
	
	\noindent Equivalently, $ \sigma \in \vAut K $ if and only if the map 
	\[ \sigma_G\colon G\to G,\quad\sigma_G(v(a)):=v(\sigma(a)) \]
	is a well defined element of $ \oAut G $, if and only if the map
	\[\bar\sigma\colon\bar K \to\bar K,\quad a+I_K\mapsto \sigma(a)+I_K \]
	is a well defined element of $ \Aut\bar{K} $.
	
	\item 
	If we fix an isomorphism \(f\colon\bar{K}\overset{\sim}{\longrightarrow} k\), the automorphism $ \bar\sigma $ uniquely defines an automorphism $ \sigma_k\in\Aut k $, given by $ \sigma_k = f \bar\sigma  f\inv $.
	Note that, if we write $ f = \rho_f f_c $ then $ \sigma_k(x) = \rho_f(\sigma(\rho_f\inv(x))_0) $ for all $ x\in k $.
	We call $ \sigma_G $ and $ \sigma_k $ the \emph{automorphisms induced by $ \sigma $ on $ G $ and $ k $}, respectively.
	\qed
	\end{enumerate}
\end{rmk}

\begin{rmk}[The automorphism group of $ R_K $]
	Let $ \Aut R_K $ denote the automorphism group of the valuation ring.
	The map $ \vAut K \to \Aut R_K $ given by $ \sigma\mapsto\sigma|_{R_K} $ is well defined by Remark~\ref{rmk:val-pres-induces-ord}.
	Moreover, since $ R_K $ is a valuation ring in $ K $, 
	then for all $ a\in K $ there exist $ c,d\in R_K $ with $ d\neq 0 $ and $ a = c/d $. Then
	every automorphism $ \nu\in\Aut R_K $ extends uniquely to an automorphism $ \sigma\in\vAut K $ by 	
$ 	\sigma(a) = \sigma\brackets{\frac{c}{d}} = \frac{\nu(c)}{\nu(d)} $.
	We see that $ \vAut K \simeq \Aut R_K $, therefore, our work also provides a study of $ \Aut R_K $.
	\qed
\end{rmk}
 
If $ k $ is an archimedean field,  an order preserving automorphism of $ K $ also preserves the valuation $ v $, so we have $ \oAut K \leq \vAut K $ (\cite[Lemma~1.3]{salma-monograph}).
If $ k $ is not archimedean this is no longer true, as shown in Example~\ref{es:counter-non-arch} below.

\begin{notation}\label{not:hahn-sum}
	Let $ A,B $ be ordered abelian groups. We denote by $ G=A\amalg B $ the group $ A\oplus B $ endowed with the lexicographic order.
	More generally, if $ \Gamma $ is a chain and $ Q $ is an ordered abelian group, then $ H=\coprod_\Gamma Q $ is the ordered abelian group $ \bigoplus_{\gamma\in\Gamma} Q $ endowed with the lexicographic order. We denote an element $ h\in H $ by $ h = \sum_{\gamma\in\Gamma}q_\gamma\1_\gamma $ where $ q_\gamma \in Q $, $ \1_\gamma $ is the characteristic function on the singleton $ \{\gamma\} $ and $ \supp(h) = \{ \gamma\in\Gamma : q_\gamma\neq 0 \} $ is finite.
\end{notation}

\begin{es}\label{es:counter-non-arch}
	Let $ A = \coprod_{\Z^{<0}}\Q $, $ B=\coprod_{\Z^{\geq 0}}\Q $ and $ G = A\amalg B = \coprod_\Z \Q $. The group $ G $ is what is called a \emph{Hahn group}. It is a valued group with valuation $ v_G\colon G^{\neq 0 }\to \Z $ given by $ v_G(\sum_{n\in\Z} q_n\1_n) = \min\{n:q_n\neq 0\} $ (for more on Hahn groups see \cite{ours-groups, michelethesis}).
	
	\noindent
	Consider the ordered Hahn fields $ k = \R\pow{B} $, $ K = k\pow{A} $ and $ F = \R\pow{G} $. The canonical valuation $ v_{\min}^K $ on $ K $ has value group $ v_{\min}^K(K^\times)=A $ and non-archimedean residue field $ k=\R\pow{B} $.
	
	\noindent
	The canonical valuation $ v_{\min}^F $ has value group $ v_{\min}^F(F^\times) = G $ and residue field $ \bar F = \R $. This is the finest convex valuation on $ F $ (\cite[p~17]{salma-monograph}). The field $ F $ also admits a coarser valuation $ w $ whose value group is $ w(F^\times) = A =v_{\min}^K(K^\times)$ and residue field is $ \bar{F}^w = \R\pow{B}=k $. The valuation $ w $ is defined (for non-zero elements) by
	\begin{equation}\label{eq:w-valuation}
	w\left( \sum_{(q,r)\in G}\alpha_{(q,r)}t^{(q,r)}\right) = \min\{ q\in A:\exists r\in B,\ \alpha_{(q,r)}\neq 0 \}.
	\end{equation}
	It is straightforward to verify that the map
	\[
	\xi\colon (F,w)\to (K,v_{\min}^K),\quad
	\sum_{(q,r)\in G}\alpha_{(q,r)}t^{(q,r)}\mapsto
	\sum_q\left( \sum_r \alpha_{(q,r)}x^r \right)y^q,
	\]
	where $ x $ is the variable in $ \R\pow{B} $ and $ y $ the variable in $ k\pow{A} $,
	is an isomorphism of valued fields (i.e., for all $ \alpha \in F $ we have $ w(\alpha)=v_{\min}^K(\xi(\alpha)) $).
	
	\noindent
	We will construct an order preserving automorphism $ \sigma $ of $ F $ that does not preserve $ w $ and therefore $ \xi\inv \sigma \xi $ will be an order preserving automorphism of $ K $ that does not preserve $ v_{\min}^K $.
	
	\noindent
	Consider the automorphism $ \rho $ of the chain $ (\Z,<) $ given by $ n\mapsto n+1 $. Since $ G $ has the canonical first lifting property as a Hahn group (see ~\cite{ours-groups}), the map $ \rho $ lifts to an order preserving automorphism $ \rho_G\in\oAut G $ given by $ \rho_G\left( \sum_{n\in\Z}q_n\1_n \right) = \sum_{n\in\Z}q_n\1_{\rho(n)} $. 
	Now, since $ F $ is a maximal Hahn field, by Example~\ref{canonical-puiseux-kappa} it has the canonical first lifting property, so $ \rho_G $ lifts to an automorphism $ \sigma\in v_{\min}^F\text{-}\Aut F $. Since $ G $ is divisible, $ F $ is real closed and all its automorphisms are necessarily order preserving, so $ \sigma\in\oAut F $ \cite[Theorem~8.6]{priess-crampe}.
	Now we show that $ \sigma $ does not preserve $ w $.
	
	\noindent
	Set $ U_w = \{ a\in F^\times : w(a)=0 \} $, the group of units of the valuation ring of $ w $ and $ G_w = v_{\min}^F(U_w) $. From \eqref{eq:w-valuation}, it is straightforward to verify that $ G_w = B $. Write $ \Gamma_w = v_G(G_w^{\neq 0}) = \Z^{\geq 0} $. Then \cite[Theorem~4.7]{kuhlmann-matusinski-point} implies that $ \sigma $ preserves $ w $ if and only if the induced chain automorphism $ \rho $ of $ \Z $ preserves $ \Gamma_w $. But we have $ \Gamma_w = \Z^{\geq 0}\neq \Z^{>0}=\rho(\Gamma_w) $ thus $ \sigma $ does not preserve $ w $.
	\qed
\end{es}

\section{Decomposition theorems}\label{section:decomposition}

\subsection{The first lifting property}\label{sec:lifting-property}

Let $ K $ be a Hahn field.
As announced in the introduction, our aim is to study $ \vAut K $.
Let $\sigma\in \vAut K$, $\sigma_G \in \oAut G$ the induced automorphism of the value group and $ \bar \sigma \in \Aut \bar K $ the one induced on the residue field.
This gives rise to a map

\begin{equation}\label{eq:Phi_K}
\Phi_K\colon  \vAut K  \longrightarrow  \Aut \bar K \times \oAut G,\qquad
\sigma \longmapsto (\bar \sigma,\sigma_G).
\end{equation}

Let \(f\colon\bar{K}\overset{\sim}{\longrightarrow} k\) be an isomorphism. We recall that the automorphism $ \bar\sigma $ uniquely defines an automorphism $ \sigma_k\in\Aut k $, given by $ \sigma_k = f \bar\sigma  f\inv $ (see  Remark~\ref{rmk:val-pres-induces-ord}).
This defines a map
\begin{equation}\label{eq:Phi}
\Phi_{K,f}\colon  \vAut K  \longrightarrow  \Aut k \times \oAut G,\qquad
\sigma \longmapsto (\sigma_k,\sigma_G).
\end{equation}
It is straightforward to verify that $ \Phi_{K,f} $ is a group homomorphism and that, if $ e\colon \bar K \to k $ is another isomorphism, then $ \Phi_{K,e} $ and $ \Phi_{K,f} $ are related by the formula
\begin{equation}\label{eq:Phi-transf-formula}
\Phi_{K,e} = \Phi_{K,\delta f}\quad \text{where $ \delta:=ef\inv \in \Aut k $.}
\end{equation}

\medskip
\noindent\textbf{Whenever the context is clear we will omit $ K,f $ from the notation and write $ \Phi $ instead of $ \Phi_{K,f} $.}
\begin{defn}
	\label{def-int-field}
	The kernel $ \ker\Phi $ of the map \eqref{eq:Phi}
	is a normal subgroup of $ \vAut K $ that we call the
	subgroup of \emph{internal automorphisms of} $K$.
	We will denote it  by $\Int\Aut K$.
	We write $ \Int\Aut_{(k)} K:= \Int\Aut K \cap\, \vAut_{(k)}K $ and
	$ \Int\Aut_{k} K:= \Int\Aut K \cap\, {\vAut_{k}K} $.
	Notice that, since $ \sigma_k = f\bar\sigma f\inv $, then $ \sigma_k = \id_k \iff \bar\sigma=\id_{\bar K} $. The definition of $ \Int\Aut K $ is therefore independent of our choice of $ f $.
\end{defn}

\noindent
We explicitly point out the following key properties of internal automorphisms:
\begin{prop}\label{rmk:internal-key-property}
	 Let $ \sigma\in\vAut K $. Then $ \sigma\in\Int\Aut K $ if and only if both of the following hold:
	 \begin{enumerate}[(i)]
	 	\item $ v(a) = v(\sigma(a)) $, for all $ a\in K $;
	 	\item if $ a\in R_K $ then $ \sigma(a)_0 = a_0 $.
	 \end{enumerate}
 Moreover, we have $ \Int\Aut_{(k)}K = \Int\Aut_kK $.
\end{prop}
\begin{proof}
	Let $ \sigma\in \Int\Aut K $. Then
	(i) holds by definition, because $ \sigma_G = \id_G $.
	To prove (ii) let $ a\in R_K $ and let $ \bar\sigma\in \Aut \bar K $ be defined as in Remark~\ref{rmk:val-pres-induces-ord}. Then, since $ \bar\sigma = \id_{\bar K} $ we have
		\[	
		a_0 + I_K = a+I_K = \bar\sigma (a + I_K) = \sigma(a)+I_K = \sigma(a)_0 +I_K
		\]
	which, since $ a_0, \sigma(a)_0 \in k $, implies $ a_0= \sigma(a)_0 $.
	
Vice versa, let $ \sigma\in\vAut K $ satisfy (i) and (ii).
From (i) it follows that $ \sigma_G=\id_G $. Let $ a\in R_K $ and compute
\[
\bar\sigma(a+I_K)=\sigma(a)+I_K = \sigma(a)_0 +I_K \overset{(ii)}{=} a_0+I_K.
\]
Thus $ \bar\sigma=\id_{\bar K} $ and so $ \sigma\in\Int\Aut K $.
The last statement now follows immediately.
\end{proof}

In Section~\ref{subsec:internal} we will study the group $ \Int\Aut K $ of internal automorphisms of a Hahn field $ K $. Now we want to determine a complement of $ \Int\Aut K $ in $ \vAut K $.

\begin{defn}\label{def-lifting-fields}
	We say that a pair
	$ (\rho,\tau) \in \Aut k\times \oAut G $ \emph{lifts to} $ K $ if there exists
	an automorphism $ \sigma\in\vAut K $ such that $ \Phi(\sigma)=(\rho,\tau) $. In this case we call $ \sigma $ a \emph{lift of} $ (\rho,\tau) $.
	If the map $ \Phi $ defined in \eqref{eq:Phi} admits a section (in particular, $ \Phi $ is surjective), i.e., an injective group homomorphism $ \Psi\colon \Aut k \times \oAut G\to \vAut K $ such that $ \Phi \Psi = \id $, then every pair
	$ (\rho,\tau) \in \Aut k\times \oAut G $ lifts to an automorphism $ \Psi(\rho,\tau) $ of $ K $ and  we say that $ K $ has 	the \emph{first lifting property with respect to $ \Psi $}.
\end{defn}

\begin{es}
	\begin{enumerate}[(i)]
		\item The maximal Hahn field $ \K $ has the first lifting property. Indeed, for every pair $ (\rho,\tau)\in\Aut k\times \oAut G $ the map $ \sigma\colon\sum a_gt^g\mapsto \sum\rho(a_g)t^{\tau(g)} $ is an automorphism of $ \K $ such that $ \Phi(\sigma)=(\rho,\tau) $ (see Corollary~\ref{canonical-puiseux-kappa}).
		\item The field $ k(G) $ has the first lifting property. The map $ \sigma $ of part (i) restricts to an automorphism of $ k(G) $.
		\item A large class of Hahn fields with the first lifting property will be described in Section~\ref{subsec:rayner}.
	\end{enumerate}\qed
\end{es}

\begin{rmk}
	The morphism $ \Phi_K $ admits a section if and only if $ \Phi_{K,f} $ admits a section, for every isomorphism $ f\colon \bar K \to k $.
	Indeed, if $ e\colon \bar K \to k $ is another isomorphism and $ \Psi_{K,f} $ is a section of $ \Phi_{K,f} $, then it follows from Equation~\eqref{eq:Phi-transf-formula} that a section $ \Psi_{K,e} $ of $ \Phi_{K,e} $ is given by the formula
	\begin{equation}\label{eq:Psi-transf-formula}
	\Psi_{K,e}(\rho,\tau) = \Psi_{K,f}(\delta\inv\rho\delta,\tau)
	\quad
	\text{ where } \delta = ef\inv.
	\end{equation}
	\qed
\end{rmk}

\begin{defn}
	\label{def-ext-field}
	Assume that $ K $ has the first lifting property with respect to a fixed section $ \Psi $ of $ \Phi $.
	The subgroup $ \Psi(\Aut k \times \oAut G) $ of $ \vAut K $ will be called the subgroup of $ \Psi $\emph{-external automorphisms of $K$} and denoted by $\pExt\Aut K$.
	Therefore, for every $ \sigma\in\Psi\text{-}\Ext\Aut K $ and every $ (\rho,\tau)\in\Aut k \times \oAut G $ we have
	\[
	\Phi(\sigma) = (\rho,\tau)\iff \Psi(\rho,\tau)=\sigma.
	\]
	Hence, for any section $ \Psi $ we have $ \Psi\text{-}\Ext\Aut K \simeq \Aut k \times \oAut G $. 
	
	\noindent
	We will also use the notations $ \pExt\Aut_{(k)} K:= \pExt\Aut K \cap \Aut_{(k)}K $ and
$ \pExt\Aut_{k} K:= \pExt\Aut K \cap \Aut_{k}K $.
\end{defn}

\begin{rmk}\label{rmk:Psi1-Psi2}\label{rmk:lifts-k-stable}
	\begin{enumerate}[(i)]
	\item 
	Let $ \sigma\in\Psi\text{-}\Ext\Aut K $, say $ \sigma =\Psi(\rho,\tau) $ for $ (\rho,\tau)\in\Aut k \times \oAut G $. Then $ \sigma\in\vAut_k K \iff \rho=\id_k $.
	Thus  
	$ \pExt\Aut_k K\simeq \oAut G $.
	\item 
	If $ K $ has the first lifting property with respect to $ \Psi $,
	the homomorphism $ \Phi $ defined in \eqref{eq:Phi} is surjective and every pair $ (\rho,\tau)\in\Aut k \times \oAut G $ lifts to an automorphism $ \sigma=\Psi(\tau)\in \vAut K $.
	The first isomorphism theorem then yields
	\[
	\Aut k \times \oAut G \simeq \frac{\vAut K }{\Int\Aut K}.
	\]
	In particular, the set of lifts of some pair $ (\rho,\tau)\in\Aut k \times \oAut G $ is the coset $ \{ \sigma \sigma' : \sigma'\in\Int\Aut K \} $, where $ \sigma $ is any given lift of $ (\rho,\tau) $.
	\item 
	Notice also that a $ k $-automorphism is not necessarily internal. Let $ \id_G\neq\tau\in\oAut G $, then the the pair $ (\id_k,\tau) $ lifts to an automorphism $ \sigma\in \vAut_k K \setminus \Int\Aut K. $\qed
	\end{enumerate}
\end{rmk}

\begin{cor}\label{prop:external-direct}
	We have:
	\begin{align}
		\pExt\Aut K &\simeq \Aut k \times \oAut G \label{eq:external-direct-k-stable}\\
		\pExt\Aut_kK &\simeq \oAut G \label{eq:external-direct-k}
	\end{align}\qed
\end{cor}
Analogous results to Corollary~\ref{prop:external-direct} for $ k $-stable automorphisms will appear in Section~\ref{sec:CLP}.
\noindent\textbf{}

\subsection{The group of valuation preserving automorphisms}
\label{subsec:Hofberger}
In this section we assume that all Hahn fields under consideration have the first lifting property with respect to $ \Psi $. Whenever the context is clear we will omit $ \Psi $ from the notation and terminology.
Recall our notation $ \K = k\pow{G} $.
In \cite[Satz 2.2]{hofbergerthesis} Hofberger shows that $ \vAut \K $ can be decomposed into a semi-direct product of the groups of internal and
external automorphisms (with respect to a specific section -- see \eqref{eq:Phi-can}).
We generalise Hofberger's result to a Hahn field $ K\subseteq \K $ which has the first lifting property with respect to an arbitrary section.

\begin{theorem}
	\label{hofberger-semi-direct}
	Let $K\subseteq \K$ be a Hahn field with the first lifting property.
	Then we have the following inner \footnote{
	For $ \alpha_i\in \Int\Aut K $ and $ \beta_i\in \Ext\Aut K $, $ i=1,2 $, we have $ (\alpha_1,\beta_1)(\alpha_2\beta_2)=(\beta_1\inv\alpha_1\beta_1\alpha_2,\beta_1\beta_2) $.
	} 
	semi-direct product decompositions:
\begin{alignat}{3}
	\vAut K &= \Int\Aut K &&\rtimes \Ext\Aut K \label{eq:hofberger-plain}\\
	\vAut_{(k)} K &= \Int\Aut_{k} K &&\rtimes \Ext\Aut_{(k)} K \label{eq:hofberger-k-stable}\\
		\vAut_k K &= \Int\Aut_k K &&\rtimes \Ext\Aut_k K
		 \label{eq:hofberger-k}
\end{alignat}
\end{theorem}
\begin{proof}
	Let $ \Phi $ be defined as in \eqref{eq:Phi} and let $ \Psi $ be a section.
	Consider the sequence
	$$
	\xymatrix{ 
		\Int\Aut K \ar[r]^-\iota &
		\vAut K \ar[r]^-{\Phi_f} &
		\Aut k \times \oAut G \ar@{-->}@/^1pc/[l]^\Psi
	}
	$$
	where $\iota$ is the canonical embedding. By definition of $\Int\Aut K$ we have $\im\iota = \ker\Phi_f$
	so the sequence is exact. Therefore (see~\cite[p.~109]{maclane-homology}) we have:
	\[	\vAut K= \im\iota\rtimes \im\Psi 
	= \Int\Aut K \rtimes \Ext\Aut K.\]
	Equations~\eqref{eq:hofberger-k-stable} and~\eqref{eq:hofberger-k} are obtained from \eqref{eq:hofberger-plain} by taking intersections with $ \vAut_{(k)}K $ and $ \vAut_kK $ respectively and using Proposition~\ref{rmk:internal-key-property}.
\end{proof}

Let $ K $ have the first lifting property. By Theorem~\ref{hofberger-semi-direct}, describing $ \vAut K $ consists of two tasks: describing the normal subgroup $ \Int\Aut K $ and the subgroup $ \Ext\Aut K $. By Corollary~\ref{prop:external-direct}, we have $ \Ext\Aut K\simeq \Aut k \times \oAut G $. A detailed study of $ \oAut G $ in terms of its value set $ v_G(G^{\neq 0}) $ is carried out in \cite{ours-groups}. In the next section we investigate the structure of $ \Int\Aut K $.

\subsection{The canonical second lifting property}\label{subsec:internal}
In this section we study the group $ \Int\Aut K $ in more detail and provide a decomposition into a semi-direct product of two notable subgroups. As announced in part~(iv) of Remark~\ref{coeff-isom} we work with the following:

\begin{defn}\label{def:G-exp}
	Let $ \Hom(G,k^\times)$ be the set of all homomorphisms of the additive group $ (G,+) $ into the multiplicative group $ (k^\times,\cdot) $. Let $ x\in  \Hom(G,k^\times)  $.
	We will denote the image of a $ g\in G $ under $ x $ by $ x^g:=x(g) $, so, for all $ g,h\in G $ we have $ x^{g+h}=x^gx^h $.
	Let $ \mathbf1\colon(G,+)\to(k^\times,\cdot),\ g\mapsto 1 $ be the trivial morphism.
	Then the set $  \Hom(G,k^\times)  $ forms a group under the pointwise multiplication defined by $ (xy)^g:=x^gy^g $.
	The inverse of $ x $ is the morphism $ g\mapsto x^{-g} $ and $ \mathbf1 $ is the neutral element.
\end{defn}

\begin{lemma}\label{X-lemma}
	Let $ K $ be a Hahn field. The map
\begin{equation}\label{eq:Xi}
	X\colon\Int\Aut K \to\Hom(G,k^\times) ,\quad \sigma\mapsto x_\sigma
\end{equation}
where $ x_{\sigma} $ is defined by
\begin{equation}\label{eq:x_sigma}
	x_{\sigma}(g):=x_{\sigma}^g:=\sigma(t^g)_g
\end{equation}
is a group homomorphism.
\end{lemma}
\begin{proof}
	Let $ \sigma\in \Int\Aut K $ and $ g,h\in G $.
	The formula \eqref{eq:x_sigma} defines an element of $ \Hom(G,k^\times) $. Indeed, by Proposition~\ref{rmk:internal-key-property}, $ g=v(t^g)=v(\sigma(t^g)) $, so $ x_\sigma(g)=\sigma(t^g)_g\in k^\times $. 
	Moreover, we have
	 \[
	 x_\sigma(g+h)=\sigma(t^{g+h})_{g+h} = (\sigma(t^g)\sigma(t^h))_{g+h} = \sigma(t^g)_g\sigma(t^h)_h=x_\sigma(g)x_\sigma(h).
	 \]
	 So $ X $ is a well defined map. To show that it is also a group homomorphism, let $ \sigma,\tau\in\Int\Aut K $. Let $ \alpha:=\tau(t^g)_g $. Then
	 \[
	 x_{\sigma\tau}(g) = (\sigma\tau(t^g))_g=\sigma(\alpha t^g)_g = \alpha\sigma(t^g)_g = x_\sigma(t^g)x_\tau(t^g).
	 \]
\end{proof}

\begin{defn}\label{def:1-aut}
	The kernel $ \ker X $ will be called the group of \emph{1-automorphisms of $ K $} and denoted by $ \uAut K $. Hence $ \uAut K \unlhd \Int\Aut K $.
	We use the notations $ \uAut_{(k)}K :=\uAut K \cap \Aut_{(k)} K $ and $ \uAut_kK:=\uAut K \cap \Aut_kK $.
\end{defn}

\begin{lemma}\label{lemma:1Aut-normal}\label{lemma:1-Aut-k=1-Aut(k)} Let $ K $ be a Hahn field. The following hold: 
	\begin{enumerate}[(i)]
		\item For all $ \tau\in\uAut K $ and all $ a\in K^\times $ we have $ \tau(a)_{v(a)} = a_{v(a)} $;
		\item $ \uAut_{(k)}K = \uAut_k K $.
	\end{enumerate}
\end{lemma}
\begin{proof}
	\begin{enumerate}[(i)]
		\item By definition of $ \uAut K $, if $ \tau\in\uAut K $ and $ a\in K^\times $ with $ v(a)=h $ we have $ \tau(a) = \tau\brackets{a_ht^h + b} $ with $ v(b)>h $. Therefore
		$ \tau(a) = \tau(a_ht^h)+\tau(b)= a_ht^h +c +\tau(b) $ for some $ b,c\in K $ with $ v(\tau(b))>h $ and $ v(c)>h $.
		\item 
		An automorphism that fixes the first coefficient of every series and keeps $ k $ invariant is necessarily trivial on $ k $.
	\end{enumerate}
\end{proof}

\begin{defn}\label{def:G-exp-closed}
	We say that a Hahn field $ K $ satisfies the  \emph{canonical second lifting property} if the map
	\begin{equation}\label{eq:embed-E(G,k)-in-Int}
		P\colon \Hom(G,k^\times) \rightarrow\Int\Aut K,\quad x\mapsto\rho_x
	\end{equation}
	where $ \rho_x $ is given by 
	\begin{equation}\label{eq:rho-x-formula}
		\rho_x(\sum a_gt^g) = \sum a_gx^gt^g
	\end{equation}
is a well defined group homomorphism.
\end{defn}

\begin{prop}\label{prop:embed-G-Exp}\label{rmk:G-Exp-properties-hom}
	
Let $ K $ satisfy the canonical \footnote{In analogy to Definition~\ref{def-lifting-fields} one could define a \emph{general} second lifting property. However, in this paper we only work with the canonical second lifting property.} second lifting property. Then the map $ P $ of Definition~\ref{def:G-exp-closed} is injective and a section of $ X $, that is $ XP = \id_{ \Hom(G,k^\times) } $. In particular, $ X $ is surjective.

\end{prop}
\begin{proof}
	Let $ x,y\in  \Hom(G,k^\times)  $ and let $ a = \sum a_gt^g\in K $. If $ \rho_x = \rho_y $ then, for all $ g\in G $ we have $ x^gt^g = \rho_x(t^g) = \rho_y(t^g) = y^gt^g $,
	which implies $ x^g=y^g $ for all $ g $ and so $ x=y $. So $ P $ is injective.
	
	\noindent
	Moreover, for all $ g\in G $ we have
	\[
	X P(x)(g)=X(\rho_x)(g) =(\rho_x(t^g))_g = (x^gt^g)_g=x^g,
	\]
	thus $ XP(x)=x $ which proves $ X P = \id_{ \Hom(G,k^\times) } $ and, in particular, $ X $ is surjective.
\end{proof}

\begin{defn}\label{def:g-exp}
	Let $ K $ satisfy the canonical second lifting property.
	The subgroup $ \{\rho_x:x\in \Hom(G,k^\times) \}=\im  P \leq \Int\Aut K $ is called the group of \emph{$ G $-exponentiations on $ K $} and denoted by $ \gexp{G}K $.
	\\
	Clearly we have $ \gexp{G}K\simeq \Hom(G,k^\times) $, hence it only depends on $ G $ and $ k^\times $.
\end{defn}

\begin{rmk}\label{rmk:G-Exp-properties}
Let $ K $ be a Hahn field satisfying the canonical second lifting property. The following assertions hold.
	\begin{enumerate}[(i)]
		\item We could define a notion of \emph{general second lifting property} along with the canonical one, similarly to what we did for the first lifting property. We refrain from doing so in this paper and refer the interested reader to~\cite{michelethesis}.
		\item 
		Composing the two maps from Proposition~\ref{prop:embed-G-Exp} we get a homomorphism
		\[
		P X\colon\Int\Aut K \to \gexp{G}K,\ \sigma\mapsto\rho_{x_\sigma}
		\]
		that associates to an internal automorphism $ \sigma $ its $ \gexp{G}K $ component.
		\item 
		By \eqref{eq:rho-x-formula}, the inverse of $ \rho_x $ is given by
		\begin{equation}\label{eq:rho-x-inv}
			\rho_x\inv\left( \sum a_gt^g \right) = \sum a_gx^{-g}t^g.
		\end{equation}
		\item All $G$-exponentiations are trivial on $ k $, so we have $ \gexp{G}K\leq \Aut_kK $.
		\item For all $ \rho\in\gexp{G}K $ and for all $ a\in K $ we have $ \supp\rho(a) = \supp a $.
		\qed
	\end{enumerate}
\end{rmk}

\begin{ess}\label{can-sec-LP:examples}
	\begin{enumerate}[(i)]
		\item The maximal Hahn field $ \K $ satisfies the canonical second lifting property. Indeed, for all $ a = \sum a_gt^g\in \K $ and all $ x\in\Hom(G,k^\times) $ the element $ \rho_x(a)=\sum a_gx^gt^g $ has the same support as $ a $ (see Remark~\ref{rmk:G-Exp-properties}(iv)), hence $ \rho_x(a)\in \K $ and $ P $ is well defined, as required.
		\item 
		The field $ k(G) $ satisfies the canonical second lifting property. Indeed, let $ a = c/d \in k(G) $ for $ c,d\in \K $ with finite support. Let $ x\in\Hom(G,k^\times) $ and consider $ \rho_x $ as an automorphism of $ \K $. Then
		\[
		\rho_x\brackets{\frac{c}{d}} = \frac{\rho_x(c)}{\rho_x(d)} \in k(G)
		\]
		because $ \supp(c)=\supp(\rho_x(c)) $ and $ \supp(d) = \supp(\rho_x(d)) $ are finite.
		\item 
		A large class of Hahn fields satisfying the canonical second lifting property will be described in Section~\ref{subsec:rayner}.
		\qed
	\end{enumerate}
\end{ess}

\noindent
The next result characterises internal automorphisms as products of $G$-exponen\-tiations and 1-automorphisms.

\begin{lemma}
	\label{lemma:int_decomp_rho_tau}
	Let $ K $ satisfy the canonical second lifting property and let
	$ \sigma\in\vAut K $. Then
	$ \sigma \in \Int\Aut K$ if and only if there exist $ \rho  \in \gexp{G}K$ and $ \tau\in \uAut K $ such that $ \sigma = \rho \tau $.
\end{lemma}
\begin{proof}
	Since $ \Int\Aut K $ is a group, a composition of internal automorphisms is internal. So if there exist $ \rho, \tau $ as in the statement, then in particular $ \rho,\tau \in \Int\Aut K $, therefore $ \sigma=\rho \tau\in\Int\Aut K $. 
	
	\noindent
	Conversely, let $ \sigma \in\Int\Aut K $.
	For all $ g\in G\ $ let $ x^g:= \sigma(t^g)_g $ be the first coefficient of $ \sigma(t^g) $. Notice that the elements $ \{x^g:g\in G\} $ have the property that
	\begin{equation}
	\label{G-exp-property}
	x^gx^h = x^{g+h}.
	\end{equation} 
	Indeed $ x^{g+h} $ is the first coefficient of $ \sigma(t^{g+h}) = \sigma(t^gt^h)=\sigma(t^g)\sigma(t^h) $ and the first coefficient of the last series is the product of the first coefficients of the factors. Hence the map $ x\colon G\to k^\times,\ g\mapsto x^g $ is an element of $  \Hom(G,k^\times)  $, and the corresponding $ \rho_x $ defined as in \eqref{eq:embed-E(G,k)-in-Int} is a $ G $-exponentiation on $ K $.
	Set $ \rho = \rho_x $ and let $ \tau := \rho\inv \sigma $.
	Obviously we have $ \sigma = \rho \tau $, so we just need to show that $ \tau \in \uAut K $.
	Let $ a\in K $ and let $ h = v(a) $.
	Then we have 
	\begin{align*}
	\tau(a)_h &= (\rho\inv\sigma(a))_h 
	= (\rho\inv\sigma(a_ht^h))_h 
	= (\rho\inv\sigma(a_h)\cdot\rho\inv\sigma(t^h))_h\\
	&= (\rho\inv\sigma(a_h))_0\cdot(\rho\inv\sigma(t^h))_h
	=a_h (x^h)\inv\sigma(t^h)_h
	=a_h.
	\end{align*}
	So $ \tau\in\uAut K $ and the proof is complete.
\end{proof}

\subsection{The group of internal automorphisms}\label{subsec:intern-decomp}
The next proposition gives a decomposition of $ \Int\Aut K $ that will be used to further refine Theorem~\ref{hofberger-semi-direct}.

\begin{theorem}\label{prop:internal-semi-direct}
	Let $ K $ satisfy the canonical second lifting property. Then
	the group $ \Int\Aut K $ (resp.\ $ \Int\Aut_kK $) admits the following semi-direct product decomposition:
\begin{align}
	\Int\Aut K &= \uAut K \rtimes \gexp{G} K \label{eq:Int-decomp-general}\\
	\Int\Aut_{(k)} K = \Int\Aut_kK &= \uAut_k K \rtimes \gexp{G} K\label{eq:Int-decomp-k-stable}
\end{align}
\end{theorem}
\begin{proof}
	Consider the sequence
		\[
	\xymatrix{
		\uAut K \ar@{^{(}->}[r]^{\iota} & \Int\Aut K \ar@{->>}[r]^{X} &  \Hom(G,k^\times)  \ar@{-->}@/^1pc/[l]^P
	}
	\]
	where $ \iota $ is the canonical embedding. By Lemma~\ref{lemma:1Aut-normal} we have $ \ker X = \uAut K = \im \iota $ so the sequence is exact and, by Proposition~\ref{prop:embed-G-Exp}, $ P $ is a section of $ X $. Hence \eqref{eq:Int-decomp-general} follows.
	
	By Remark~\ref{rmk:G-Exp-properties} we have $ \gexp{G}K\leq \vAut_kK $ and by Lemma~\ref{lemma:1Aut-normal} we have $ \uAut_{(k)}K=\uAut_k K $, so $ \eqref{eq:Int-decomp-k-stable} $ follows.
\end{proof}

\begin{rmk}
	We noted (Remark~\ref{not:general-KvG-2}) that $ \Hom(G,k^\times) $ is a direct factor of the group $ \Hom(G,U) = \Hom(G,1+I_K)\times \Hom(G,k^\times) $.
	Under some further assumptions, in Section~\ref{subsec:general} we will be able to relate $ \uAut_k K $ to $ \Hom(G,1+I_K) $ (with a twisted group operation), thereby relating $ \Int\Aut K $ to $ \Hom(G,U) $.
	\qed
\end{rmk}

\subsection{The canonical first lifting property}\label{sec:CLP}

Recall that (Remark~\ref{coeff-isom}) for any Hahn field $ K $ there is an isomorphism $ f\colon \bar K \to k $. We can choose this isomorphism canonically to be $ f_c\colon a+I_K\mapsto a_0 $, for all $ a \in R_K $.
We call $ f_c $ the \emph{canonical} or \emph{coefficient isomorphism} between $ \bar K $ and $ k $.
The homomorphism $ \Phi_{K,f} $ defined in \eqref{eq:Phi} implicitly depends on the choice of $ f $:
from now on we fix this to be the coefficient isomorphism $ f_c $.
Then $ \Phi_{K,f} $ assumes the special form
\begin{equation}\label{eq:Phi-can}
	\Phi_c\colon  \vAut K  \longrightarrow  \Aut k \times \oAut G,\qquad
	\sigma \longmapsto (\sigma_k,\sigma_G)
\end{equation}
where $ \sigma_k = f_c\bar\sigma f_c\inv $. Computing gives $ f_c\bar\sigma f_c\inv(a_0) = f_c\bar \sigma  (a_0 +I_K) = f_c(\sigma(a_0) +I_K) = \sigma(a_0)_0 $, for all $ a_0\in k $. Thus 
\begin{equation}\label{eq:sigma_k-on-k}
	 \sigma_k(a_0)=\sigma(a_0)_0 \quad \text{for all } a_0\in k.
\end{equation}

\begin{rmk}\label{rmk:induced=restricted}
	\label{rmk:k-normal-in-(k)}
	Let $ \sigma\in \vAut_{(k)}K $.
	Then $ \sigma|_k = \sigma_k $.
	Indeed, for $ a_0\in k $, from $ \sigma\in\vAut_{(k)}K $ it follows that $ \sigma(a_0)\in k $ so Equation~\eqref{eq:sigma_k-on-k} gives $ \sigma_k(a_0) = \sigma(a_0)_0 = \sigma(a_0) = \sigma|_k(a_0) $.		
	Moreover, let $ \pi_1:\Aut k \times \oAut G\to\Aut k,\ (\rho,\tau)\mapsto\rho $ be the projection on the first component.
	Then the restriction $ \pi_1\Phi_c\colon\vAut_{(k)}K\to \Aut k $ is a homomorphism with kernel $ \vAut_{k} K $.
	Thus $ \vAut_{k} K \unlhd \vAut_{(k)}K $.
	\qed
\end{rmk}

\begin{defn}\label{def:canonical-lift}
	Let $ (\rho,\tau) \in\Aut k\times\oAut G $. The automorphism $ \widetilde{\rho\tau}\in\vAut \K $ given by
	\begin{equation}\label{eq:canonical-lfit}
	\widetilde{\rho\tau}\left( \sum_{g\in G}a_g t^g \right) = 
	\sum_{g\in G} \rho(a_g) t^{\tau(g)} 
	\end{equation}
	is a lift of the pair $ (\rho,\tau) $ to $ \K $ that we call \emph{the canonical lift}. Indeed, we have $ \Phi_c(\widetilde{\rho\tau})=(\rho,\tau) $.

	\noindent
	We will denote by $ \tilde{\rho} $ the lift of $ (\rho,\id_G) $ and simply refer to it as the lift of $ \rho\in\Aut k $. Similarly for $ \tilde\tau $.
\end{defn}
\begin{rmk}
	If $ k $ is an ordered field, $ \rho \in \oAut k $ and we take the induced lexicographic ordering on $ \K $, then the canonical lift of a pair $ (\rho,\tau)\in\oAut k\times \oAut G $ preserves the lexicographic ordering on $ \K $.\qed
\end{rmk}
\begin{lemma}\label{lemma:canonical-section}
	Let $ K $ be a Hahn field such that
	\begin{equation}\label{eq:CLP}
	\forall\ (\rho,\tau)\in \Aut k\times\oAut G:\quad \widetilde{\rho\tau}(K)=K.
	\end{equation}
	Then the map $ \Psi_c\colon \Aut k\times\oAut G\to\vAut K,\ (\rho,\tau)\mapsto\widetilde{\rho\tau}|_K $ is a section of $ \Phi_c $.
\end{lemma}
\begin{proof}
	Let $ (\rho,\tau)$ and $(\rho',\tau')$ be elements of $\Aut k \times \oAut G $ and let $ \widetilde{\rho\tau},\widetilde{\rho'\tau'} $ be the respective canonical lifts. Then, for all $ a = \sum a_gt^g\in K $, we have
	
	\begin{align*}
	\Psi_c(\rho \rho',\tau \tau')(a) 
	&= \sum \rho(\rho'(a_g))t^{\tau(\tau'(g))} \\
	&= \Psi_c(\rho,\tau)\left( \sum \rho'(a_g) t^{\tau'(g)} \right) \\
	&= \Psi_c(\rho,\tau)\left( \Psi_c(\rho',\tau')\left(a\right) \right).
	\end{align*}
	
	For injectivity, let $ (\rho,\tau)\neq(\rho',\tau') $. Then
	if $ \rho(\alpha)\neq\rho'(\alpha) $ for some $ \alpha\in k $ then $ \widetilde{\rho\tau}(\alpha) = \rho(\alpha)\neq\rho'(\alpha) = \widetilde{\rho'\tau'}(\alpha) $;
	similarly, if $ \tau(g)\neq\tau'(g) $ for some $ g\in G $ then $ \widetilde{\rho\tau}(t^g)=t^{\tau(g)}\neq t^{\tau'(g)} = \widetilde{\rho'\tau'}(t^g) $.
	
	\noindent Finally, we prove that $ \Phi_c \Psi_c = \id_{\Aut k\times\oAut G} $.
	Let $ (\rho,\tau)\in\Aut K \times \oAut G $ and let $ \sigma = \Psi_c(\rho,\tau) $. Then, for all $ a=\sum a_gt^g\in K $ we have $ \sigma(a)=\sum \rho(a_g)t^{\tau(g)} $. Then $ \Phi_c(\sigma)=(\sigma_k,\sigma_G) $ where $ \sigma_k $ is defined by $ \sigma_k(a_0) = \sigma(a_0)_0 = \rho(a_0) $ and $ \sigma_G $ is defined by $ \sigma_G(v(a)) = v(\sigma(a)) = v\brackets{\sum_{g\geq v(a)}\rho(a_g)t^{\tau(g)}} = \tau(v(a)) $. So $ (\sigma_k,\sigma_G) = (\rho,\tau) $ and thus $ \Phi_c\Psi_c=\id $.
\end{proof}	
\begin{defn}\label{def:CLP}
	Let $ K $ be a Hahn field satisfying \eqref{eq:CLP}. We say that $ K $ has the \emph{canonical first lifting property} and we call $ \Psi_c $ the \emph{canonical section (on $ K $)} of $ \Phi_c $.
\end{defn}
\noindent \textbf{Whenever $ K $ has the canonical first lifting property, we will assume our chosen section to be the canonical one.}

\begin{rmk}\label{rmk:canonical-lift-k-stable}
	\begin{enumerate}
		\item 
	For every pair $ (\rho,\tau) \in \Aut k \times \oAut G $ we have $ \Psi_c(\rho,\tau)\in \vAut_{(k)}K $.
	\item 
	Assume that $ K $ has the canonical first lifting property and let $ f = \rho_ff_c \in\Aut k $ (see Remark~\ref{coeff-isom}). Then an explicit section $ \Psi_{K,f} $ of $ \Phi_{K,f} $ is given by the formula:
	\begin{align}\label{eq:fLP}
	\begin{split}
	\Psi_{K,f}({\rho,\tau}) &= \Psi_{K,c}({\rho_f\inv\rho\rho_f,\tau})\\
	\Psi_{K,f}({\rho,\tau})\left( \sum_{g\in G}a_g t^g \right) &= 
	\sum_{g\in G} \rho_f\inv\rho\rho_f(a_g) t^{\tau(g)} 
	\end{split}
	\end{align}

	\end{enumerate}\qed
\end{rmk}

\begin{lemma}\label{lemma:k-stable-semidirect}
	Let $ K $ be a Hahn field with the canonical first lifting property. Then
	\[
	\vAut_{(k)}K \simeq \vAut_{k}K\rtimes \Aut k.
	\]
\end{lemma}

\begin{proof}
	By Remark~\ref{rmk:k-normal-in-(k)} we have $ \vAut_k K = \ker \pi_1\Phi_c $, so the sequence
	\[
	\vAut_{k} K \hookrightarrow \vAut_{(k)}K \overset{\pi_1\Phi_c}{\twoheadrightarrow} \Aut k
	\]
	is exact.
	Because $ \Psi_c $ is a section of $ \Phi_c $ it follows that the map $ \Aut k \to\vAut_{(k)}K$, $ \rho\mapsto\tilde\rho = \Psi_c(\rho,\id_G) $ is a section of $ \pi_1\Psi_c $.
	The statement follows.
\end{proof}

\subsection{Rayner fields}
\label{subsec:rayner}
Now we are going to study a class of Hahn fields, which satisfy the canonical first and second lifting property.
\begin{defn}\label{def:field-family}
	Let $ G $ be a non-trivial ordered abelian group.
	A family $ \cF\neq \emptyset $ of subsets of $ G $
is said to be a \emph{field family (with respect to $ G $)} (see  \cite[Section~2]{rayner1968}) if the following six properties are satisfied:
\begin{enumerate}[(R1)]
	\item The elements of $ \cF $ are well ordered subsets of $ G $.
	\item The union of the elements of $ \cF $ generates $G$ as a group.
	\item $ A,B\in\cF\Rightarrow A\cup B\in\cF $.
	\item $ A\in\cF,\ B\subset A \Rightarrow B\in\cF $.
	\item $ A\in\cF,\ g\in G\Rightarrow A+g \in\cF $.
	\item if $ A\in\cF$ and $ A\subseteq G^{\geq 0} $ then the set of all finite sums of elements of $ A $ belongs to $ \cF $.
\end{enumerate}
\end{defn}
\begin{theorem}[{\cite[Theorem 1]{rayner1968}}]
	\label{rayner-theorem}
	If $ \cF $ is a field family then the set $ k\pow{\cF} $ of elements of $ \K $ whose support belongs to $ \cF $ is a subfield of $ \K $.
\end{theorem}

\begin{defn}\label{def:rayner}
The fields $ k\pow{\cF} $ obtained in Theorem~\ref{rayner-theorem} are Hahn fields\footnote{See \cite[Theorem 3.15]{KKS-rayner-structures}.} that will be called \emph{Rayner fields}.
\end{defn}
\noindent

\begin{ess}
	\label{es:rayner-puiseux}\label{rayner-examples-ours}
	\begin{enumerate}[(i)]
		\newcounter{rayner-examples}
		\item A general class of Rayner fields is described in \cite[Section 3]{rayner1968}. In particular, the field of Puiseux series:
		let $ \K = k\pow{\Q} $ and consider the family $ \cF $ of sets of the form $ \frac{1}{d}A $ where $ d $ is a positive integer and $ A $ is a well ordered subset of $ \Z $. It is clear that $ \cF $ is a field family and that the field $ k\pow{\cF} $ thus obtained is the field $ \bbP $ of Puiseux series (see Section~\ref{subsect-puiseux}).
		\setcounter{rayner-examples}{\value{enumi}}
\end{enumerate}
\noindent
We describe further examples of interest to us.
\begin{enumerate}[(i)]
	\setcounter{enumi}{\value{rayner-examples}}
\item 	Let $ \kappa $ be an uncountable regular cardinal.
		The family $ \cF_\kappa $ of well ordered subsets of $ G $ with cardinality smaller than $\kappa $ is clearly a field family. The resulting field, denoted by $ \K_\kappa $, is called the $ \kappa $\emph{-bounded subfield} of $ \K $. It consists of all elements of $ \K $ whose support has cardinality less then $ \kappa $ (see \cite{alling:existence-eta-alpha} or \cite{kuhlmann-shelah}). 
		
\item Consider the set $S$ of finitely generated subgroups of $G$ and let $\cF$ consist of all well ordered subsets of elements of $S$.
Then $\cF$ is a field family and thus $k\pow{\cF}$ is a Hahn field containing $k(G)$.\qed
	\end{enumerate}
\end{ess}

\begin{lemma}\label{rayner-closed-G-Exp}
Let $ K $ be a Rayner field. Then $ K $ satisfies the canonical second lifting property.
\end{lemma}
\begin{proof}
	Let $ a\in K $ and $ x\in\Hom(G,k^\times) $.
	By part (iv) of Remark~\ref{rmk:G-Exp-properties}
	we have $ \supp(a)=\supp(\rho_x(a)) $. Since $ K $ is a Rayner field this implies $ \rho_x(a)\in K $.
\end{proof}

\noindent The following proposition characterises Rayner fields with the canonical first lifting property.
\begin{prop}\label{prop:rayner-criterion}
	Let $ \cF $ be a field family and let $ F = k\pow{\cF} $ be the corresponding Rayner field. Then $ F $ has the canonical first lifting property if and only if $ \cF $ is stable under $ \oAut G $, by which we mean that if $ A\in\cF $ and $ \tau \in \oAut G $ then $ \tau(A)\in \cF $.
\end{prop}
\begin{proof}
	Assume that $ \cF $ be stable under $ \oAut G $, let $ a = \sum a_gt^g \in F $ and $ (\rho,\tau)\in\Aut k \times \oAut G $. Then the support of $ \widetilde{\rho\tau}(a)=\sum \rho(a_g) t^{\tau(g)} $ is $ \{ \tau(g): g\in\supp(a) \} = \tau(\supp(a)) \in \cF $, by assumption (since $ \supp(a)\in \cF $).
	Hence $ \widetilde{\rho\tau}(a) \in F $. So $F$ has the canonical first lifting property.
	
	Vice versa, assume that $F$ has the canonical first lifting property and let $ A \in \cF $. Take any element $ a = \sum a_gt^g\in F $ such that $ \supp(a) =A $. By assumption, for all $ (\rho,\tau)\in\Aut k\times\oAut G $ we have $ b=\widetilde{\rho\tau}(a) \in F $ so, in particular, $ \supp(b) = \tau(A)\in \cF $ and so $ \cF $ is stable under $ \oAut G $.
\end{proof}
\begin{cor}
\begin{enumerate}[(i)]
	\label{canonical-puiseux-kappa}
	\item 
	The field $ \bbP $ of Puiseux series has the canonical first lifting property.
	\item 
	The $ \kappa $-bounded subfields of $ \K $ (Example~\ref{rayner-examples-ours}) have the canonical first lifting property.
\end{enumerate}
\end{cor}
\begin{proof}
	\begin{enumerate}[(i)]
		\item Let $ \frac{1}{d}A $ be as in Example~\ref{es:rayner-puiseux} above and let $ \tau $ be an order preserving automorphism of $ \Q $. Now $ \oAut(\Q,+)\simeq (\Q^{>0},\cdot) $ so $ \tau $ is multiplication by some positive rational $ \frac{m}{n} $. Then it is clear that $ \tau(\frac{1}{d}A)=\frac{1}{dn}(mA) $ also belongs to the same field family.
		\item An order preserving automorphism of $ G $ maps any well ordered subset onto another one of the same cardinality, hence $ \cF_\kappa $ is stable under $ \oAut G $.
	\end{enumerate}
\end{proof}
\begin{es}\label{es:k(G)-has-CLP}
	\label{rmk:Hahn-non-Rayner}
	Not all Hahn fields are Rayner. Let $ k=\Q,\ G=\Z $, so $ \K = \Q\pow{t} $ and consider $ K=\Q(t) $. Then $ a:=(1-t)\inv = \sum_{n\in\N}t^n\in K $ and $ \supp a = \N $. But for $ \mathrm{exp}(t):=\sum_{n\in\N}\frac{1}{n!}t^n$ we have $ \supp \mathrm{exp}(t) = \N $ and yet $ \mathrm{exp}(t)\notin K $ (see \cite[pp 765--767]{eisenstein-band-1}).
	So $ K $ contains some but not all elements of $ \K $ with support equal to $ \N $ and hence it is not a Rayner field.
	
	However $ k(G) $ has the canonical first lifting property for every ordered abelian group $ G $. Indeed $ k(G) $ is the set of elements of $ \K $ of the form $ a = \frac{p}{q} $ where $ p\in \K$ and $q \in \K^\times $ have finite support.
	Now let $ (\rho,\tau)\in\Aut k \times \oAut G $. Then the canonical lift $ \widetilde{\rho\tau} $ is an automorphism of $ \K $, so $ \widetilde{\rho\tau}(a) = \frac{\widetilde{\rho\tau}(p)}{\widetilde{\rho\tau}(q)} $. The supports $ \supp\widetilde{\rho\tau}(p) = \tau(\supp p) $ and $ \supp\widetilde{\rho\tau}(q) = \tau(\supp q) $ are finite, so $ \widetilde{\rho\tau}(a)\in k(G) $.\qed
\end{es}

\subsection{General decomposition theorem}\label{sec:decomp-general}

Combining Theorem~\ref{hofberger-semi-direct} with Theorem~\ref{prop:internal-semi-direct} and Corollary~\ref{prop:external-direct} we get the following decomposition theorem.

\begin{theorem}\label{hofberger-enhanced}Let $ K $ be a Hahn field with the first and canonical second lifting property. Then
	\begin{alignat}{3}
		\vAut K  &\simeq (\uAut K &&\rtimes \Hom(G,k^\times))\rtimes (\Aut k \times \oAut G)	 \label{eq:hofberger-enhanced-general}\\
		\vAut_k K  &\simeq (\uAut_k K &&\rtimes \Hom(G,k^\times))\rtimes \oAut G
	\end{alignat}\qed
\end{theorem}
\noindent
The canonical first lifting property allows us to refine some of the results that we proved above under the weaker assumption of a general first lifting property.

\begin{prop}\label{hofberger-enhanced-k-stable}
	Let $ K $ be a Hahn field with the canonical first and second lifting property. Then
	\begin{enumerate}[(i)]
		\item
		$ \Ext\Aut K= \Ext\Aut_{(k)} K $
		\item\label{eq:hofberger-stable}
		$ \vAut_{(k)} K = \Int\Aut_k K \rtimes \Ext\Aut K $
		\item 
		$ \vAut_{(k)} K  \simeq (\uAut_k K \rtimes \Hom(G,k^\times))\rtimes (\Aut k \times \oAut G)$.
		\item \begin{align*}
			\vAut_{(k)}K &\simeq
			\Int\Aut_kK\rtimes\brackets{\oAut G\rtimes\Aut k}\\
			&\simeq \brackets{\Int\Aut_kK\rtimes\oAut G}\rtimes\Aut k		
		\end{align*}
	\end{enumerate}
\end{prop}

\begin{proof}
	Part (i) follows immediately from Remark~\ref{rmk:canonical-lift-k-stable}.
	Parts (ii) and (iii) are analogous to Theorems~\ref{hofberger-semi-direct} and \ref{hofberger-enhanced} respectively, for the case where $ K $ has the canonical first lifting property.
	Finally, for part (iii):
	The first line follows from part (ii) and \ref{prop:external-direct}, noticing that we have $ \oAut G\rtimes\Aut k=\oAut G\times\Aut k $.
	The second line follows by combining
	Lemma~\ref{lemma:k-stable-semidirect}, Theorem~\ref{hofberger-semi-direct} and Proposition~\ref{prop:external-direct}.
\end{proof}

\begin{cor}\label{cor:1-Aut-determines-vAut}
	Let $ k(G)\subseteq K,F\subseteq \K $ be two Hahn fields with the first and canonical second lifting property. Then
	\begin{alignat}{5}
		\vAut K &\simeq \vAut F &&\iff  \uAut K &&\simeq \uAut F\\
		\vAut_{(k)} K &\simeq \vAut_{(k)} F &&\iff  \uAut_k K &&\simeq \uAut_k F\\
		\vAut_k K &\simeq \vAut_k F &&\iff  \uAut_k K &&\simeq \uAut_k F
	\end{alignat}\qed
\end{cor}

\noindent
The following diagram summarises the information on the group structure of $ \vAut K $, for a Hahn field $ K $ with the first and canonical second lifting property. 
The double line means that the smaller group is normal in the larger and, in general, all the inclusions are strict.
\begin{equation}\label{eq:diagram}
	\xymatrix{
		&&\vAut K \ar@{-}[dr] \ar@{=}[dl]	& & \\
		&\Int\Aut  K \ar@{-}[d] \ar@{=}[dl] &	& \Ext\Aut  K\ar@{=}[dr] \ar@{=}[d] &\\
		\uAut K \ar@{-}[ddrr] & \gexp{G}K \ar@{-}[ddr] & & \Psi(\Aut k) \ar@{-}[ddl] &\Psi(\oAut G) \ar@{-}[ddll]	 \\
		& &  &  & \\
		&& \{ \id_K \} && 
	}
\end{equation}

\noindent
Analogous diagrams hold for $ \vAut_{(k)} K $ and $ \vAut_k K $.
Comparing the groups appearing in different diagrams, many open questions remain, which we intend to address in future publications.

\section{Strongly additive automorphisms}\label{sec:strongly-linear}

Inspired by the work of Schilling \cite{schilling44}, in this section we study automorphisms of a Hahn field $ K $ which have the very powerful property of commuting with infinite sums.

\begin{defn}\label{def-strongly-linear}
	Let $ A = \{a_{(i)}:i\in I\}\subseteq \K $ be a family of elements of $ \K $ indexed by a set $ I $.
	Let $ \Supp A := \bigcup_{i\in I}\supp a_{(i)} $ and for $ g\in G $ define $ S_g = \{ i\in I : g\in\supp a_{(i)} \} $.
		We say that $ A $ is \emph{summable} if
		\begin{enumerate}[(a)]
			\item \label{def:summable-well-ordered} $ \Supp A $ is well ordered;
			\item \label{def:summable-finite} for all $ g \in \Supp A $ the set $ S_g $ is finite.
		\end{enumerate}
\end{defn}
\begin{lemma}\label{lemma:summable-family}
	Let $ A = \{a_{(i)}:i\in I\}\subseteq \K $ and for all $ g\in \Supp A $ let
	$
	\mathbf{a}_g := \sum_{i\in S_g}(a_{(i)})_g.
	$
	Then \begin{enumerate}[(i)]
		\item $ A $ is summable if and only if
	\[
	\mathbf{a}=\sum_{i\in I}a_{(i)} := \sum_{g\in\Supp A}\mathbf{a}_gt^g
	\]
	is a well defined element of $ \K $.

	\item 
Assume $ A $ is summable and set $ \nu = \min\{v(a_{(i)}):i\in I\} $. Then 
\begin{enumerate}[(i)]
	\item $ v(\mathbf{a})\geq \nu $. 
	\item If $ |S_{\nu}|=1 $, i.e., $ \exists! j\in I: v(a_{(j)})=\nu $, then $ v(\mathbf{a})=\nu $ and $ \mathbf{a}_\nu = (a_{(j)})_{\nu} $.
\end{enumerate}
	\end{enumerate}
	\qed
\end{lemma}

\begin{defn}\label{def:sum-of-summable}
Let $ A = \{a_{(i)}:i\in I\}\subseteq \K $ be a family of elements of $ \K $ indexed by a set $ I $.
	\begin{enumerate}[(i)]
		\item If $ A $ is a summable family we call $ \mathbf{a}=\sum_{i\in I}a_{(i)} $ \emph{the sum of} $ A $.
		\item 
		Let $ K $ be a Hahn field and $ A = \{a_{(i)}:i\in I\}\subseteq K $ a summable family. We say that $ A $ is $ K $-\emph{summable}  if $ \mathbf{a} \in K $.
		\item 
		Let $ C=\{c_i:\ i\in I\} \subseteq k $ be a family of coefficients. We define the family $ CA = \{c_ia_{(i)}:i\in I\}\subseteq \K $. We call $ CA $ the \emph{scalar multiple}  of $ A $ by $ C $.
		\item 
		Let $B=\{b_{(i)}:i\in J\}\subseteq \K $ be another family.
		Assume, without loss of generality, that $ I=J $. We define the \emph{sum} $ A+B = \{ a_{(i)}+b_{(i)}:i\in I \} $
		and the \emph{product} $ AB=\{a_{(i)}b_{(j)}:i,j\in I\} $.		
	\end{enumerate}
\end{defn}

\begin{rmk}
	\begin{enumerate}[(i)]
		\item 
		It follows from Lemma~\ref{lemma:summable-family} that a scalar multiple of a summable family, the sum of two summable families and the product of two summable families are all summable.
		\item 
		The maximal Hahn field $ \K $ is the only Hahn field that is closed under taking sums of arbitrary summable families. Indeed, if $ K $ is a Hahn field such that every summable family $ A \subseteq K$ is $ K $-summable, then for all $ a=\sum a_gt^g\in \K $ the family $ \{a_gt^g:g\in\supp a\} $ is $ K $-summable.
		Thus $ a\in K $ and so $ K = \K $.
\qed	\end{enumerate} 
\end{rmk}
\begin{defn}
	\label{def:strongly-linear}
	Let $ K $ be a Hahn field
	\begin{enumerate}[(i)]
		\item A map $ \sigma\colon K\ra K $ is $ K $-\emph{summable}
	if, for all $ a = \sum a_g t^g \in K $,
\begin{enumerate} 
		\item \label{def:strongly-linear-summable} the family $ \{\sigma(a_gt^g):g\in\supp(a)\} $ is $ K $-summable;
		\item \label{def:strongly-linear-additive} $ \sigma(a)= \sum \sigma(a_g t^g) $.

\end{enumerate}
\item An automorphism  $ \sigma\in \vAut K $ is \emph{strongly additive} if both $ \sigma $ and $ \sigma\inv $ are $ K $-summable maps.
	\end{enumerate}
\end{defn}

\begin{notation}\label{not:str}\label{not:str-int}
	The set of strongly additive, valuation preserving automorphisms will be denoted by $ \vAut^{\Str}K $. We will also use the superscript `` $ \Str $ '' on the other groups of automorphisms, to denote the corresponding subset of strongly additive automorphisms: for example $ \Int\Aut^{\Str} K = \Int\Aut K\cap\vAut^{\Str}K $.
\end{notation}

\begin{rmk}
	\label{rmk:str-lin-first-properties}
	\begin{enumerate}[(i)]		
	\item A strongly additive automorphism needs not be valuation preserving: the non-valuation preserving automorphism constructed in Example~\ref{es:counter-non-arch} is strongly additive.
	\item A strongly additive automorphism needs not be a $ k $-automorphism: let $ \alpha\in\Aut k $ be a non-trivial automorphism of $ k $ (for example, we can choose $ k =\Q(\sqrt{2}) $ and $ \alpha\colon \sqrt{2}\mapsto -\sqrt{2} $) and let $ \sigma\in \Aut \K $ be defined by $ \sigma(\sum a_gt^g) = \sum \alpha(a_g)t^g $. This is a strongly additive automorphism that does not fix $ k $.\qed
\end{enumerate}
\end{rmk}
\noindent
The following is an example of an automorphism that is not strongly additive.
We wish to thank L.~S.~Krapp for suggesting the idea.
\begin{es}
\label{ex-non-strong}
Let $ \omega $ be the first infinite ordinal, let $ G=\coprod_{\omega+1}\Q $ (see  Notation~\ref{not:hahn-sum}) and let $ \K = \C\pow{G} $.
Elements of $ G $ have the form $ \sum_{n\in\N}q_n\1_n + q_\omega\1_\omega $ and the set $ \{\1_n:n\in\N\}\cup\{\1_\omega\} $ is a $ \Q $\textit{-valuation basis} for $ G $ (see~\cite[Page~4]{salma-monograph}).

The set $ \{t^{-\1_n}:n\in\N\}$ is algebraically independent over $ \C $ (see \cite[Theorem~3.4.2]{engler-prestel}), therefore, it extends to a transcendence basis $ \cB $ of $ \K $ over $ \C $. The set $ \{ t^{-\1_n}+t^{\1_\omega}:n\in\N \} $ is also algebraically independent so it also extends to a transcendence basis $ \cB' $. There exists a bijection $ f\colon\cB\ra\cB' $ such that $ f(t^{-\1_n}) = t^{-\1_n}+t^{\1_\omega} $ for all $ n\in\N $. Since $ \C $ is algebraically closed, a bijection of transcendence bases extends, in turn, to an isomorphism of fields: that is, there exists a field automorphism $ f\in\Aut \K $ such that $ f(t^{-\1_n}) = t^{-\1_n}+t^{\1_\omega} $ for all $ n\in\N $. 

Now consider the element $ a = \sum_{n\in\N}t^{-\1_n} $. Notice that the sequence $ (\1_n)_{n\in\N} $ is anti-well ordered in $ G $, so the support $ (-\1_n)_{n\in\N} $ of $ a $ is well ordered, so $ a\in \K $. 
However, the family $ A = \{ f(t^{-\1_n}):n\in\N \} $ is not summable.
Indeed $ \Supp A = \{ -\1_n : n\in\N \}\cup\{\1_\omega\} $ and $ \1_\omega \in \supp f(t^{-\1_n}) $ for all $ n\in \N $ hence $ S_{\1_\omega} = \N $ violates condition \eqref{def:summable-finite} of Definition~\ref{def-strongly-linear}.\qed
\end{es}

\begin{prop}
	\label{external-are-strongly-lin}
	\label{rmk:preparation-to-str-lin}
	\label{rmk:g-exp-str.lin-inv}
	Let $ K $ be a Hahn field.
	Then
\begin{enumerate}[(i)]
	\item
	$ \vAut^{\Str} K $ is a subgroup of $ \vAut K $.
\item Assume that $ K $ has the first lifting property. Then
$\Ext\Aut K \leq \vAut^{\Str} K $.
\item Assume, moreover, that $ K $ satisfies the canonical second lifting property. Then we have
$ \gexp{G} K\leq \Int\Aut_k^{\Str} K $. 
\end{enumerate}
\end{prop}
\begin{proof}
	\begin{enumerate}[(i)]
		\item 
		Let $ \sigma $ and $ \tau $ be two strongly additive automorphisms. Then, for all $ a = \sum a_gt^g\in  K $, applying subsequently the strong additivity of $ \sigma $ and $ \tau $ we get
		\[
		(\sigma \tau)(a)=\sigma\left( \sum \tau(a_gt^g) \right) = \sum  \sigma(\tau(a_gt^g))
		\]
		so $ \sigma \tau $ is strongly additive.    
		\item Let $ \sigma \in \Ext\Aut  K $. Then there are $ \rho\in\Aut k $ and $ \tau\in\oAut G $ such that, for all $ a = \sum a_gt^g \in  K $ we have $ \sigma(a) = \sum \rho(a_g)t^{\tau(g)} $. In particular, for every term $ a_gt^g $ we also have $ \sigma(a_gt^g)= \rho(a_g)t^{\tau(g)} $.
		Thus
		\[\sigma\left(\sum_{g\in G}a_gt^g\right) = \sum_{g\in G}\rho(a_g)t^{\tau(g)} = \sum_{g\in G}\sigma(a_g t^g).
		\]
		So $ \sigma $ is strongly additive.
		\item 
		By \eqref{eq:rho-x-formula} and \eqref{eq:rho-x-inv}, a $ G $-exponentiation $ \rho_x $ on $  K $ is a strongly additive automorphism.
	\end{enumerate}
\end{proof}

\subsection[Structure of the group of strongly additive automorphisms]{The structure of $ \boldsymbol{\vAut^{\Str} K}$}

\label{subsec:structure}
Since normality is preserved by taking intersections, it follows that $ \Int\Aut ^{\Str} K = \Int\Aut  K\cap\vAut ^{\Str} K \unlhd \vAut^{\Str} K $.
Decomposition results analogous to those obtained in Section~\ref{sec:lifting-property} hold for the group of strongly additive automorphisms and its subgroups:
\begin{prop}\label{prop:hofberger-str-linear}
Let $ K $ be a Hahn field with the first lifting property. Then
\begin{equation}
\vAut ^{\Str}  K = \Int\Aut ^{\Str}  K \rtimes \Ext\Aut   K.
\end{equation}	
\end{prop}
\begin{proof}
	Let $ \sigma\in\vAut ^{\Str} K $. By Theorem~\ref{hofberger-semi-direct} there exist $ \tau\in\Ext\Aut   K $ and $ \rho\in\Int\Aut  K $ such that $ \sigma = \rho\tau $. Now, by Remark~\ref{external-are-strongly-lin}, $ \tau $ is strongly additive, and since $ \rho = \sigma\tau\inv $ then $ \rho $ is also strongly additive. So $ \Ext\Aut  K $ and $ \Int\Aut ^{\Str} K $ generate $ \vAut ^{\Str} K $. Moreover, $ \Ext\Aut  K \cap\Int\Aut ^{\Str} K\subseteq \Ext\Aut  K \cap\Int\Aut  K = \{\id_ K\} $. Finally, as remarked above, $ \Int\Aut ^{\Str} K $ is a normal subgroup of $ \vAut ^{\Str} K $. The statement follows.
\end{proof}

\begin{lemma}\label{int-str-lin:lemma}
Let $ K $ be a Hahn field satisfying the canonical second lifting property. The following hold.
\begin{eqnarray}
\Int\Aut ^{\Str} K &= \uAut ^{\Str} K\rtimes\gexp{G} K;\\
\Int\Aut_{(k)} ^{\Str} K =\Int\Aut_k ^{\Str} K &= \uAut_k ^{\Str} K\rtimes\gexp{G} K.\label{eq:Int_k^{>0}-semidir}
\end{eqnarray}	
\end{lemma}

\begin{proof}
	Follows from Theorem~\ref{prop:internal-semi-direct} taking intersections with $ \vAut^{\Str} K $ and applying part~(iii) of Proposition~\ref{external-are-strongly-lin}.
\end{proof}

Finally we have

\begin{theorem}\label{eq:IntStr-semi-direct}\label{rmk:1-Aut-is-normal}\label{thm:hofberger-enhanced-str}
Let $ K $ be a Hahn field with the first and canonical second lifting property.
Then
\begin{align*}
\vAut ^{\Str}  K &= (\uAut ^{\Str} K\rtimes\gexp{G} K) \rtimes \Ext\Aut   K\\
&\simeq (\uAut ^{\Str} K\rtimes\Hom(G,k^\times)) \rtimes (\Aut k \times \oAut G);\\
\vAut_{(k)} ^{\Str}  K &= (\uAut_k ^{\Str} K\rtimes\gexp{G} K) \rtimes \Ext\Aut   K\\
&\simeq (\uAut_k ^{\Str} K\rtimes\Hom(G,k^\times)) \rtimes (\Aut k \times \oAut G);\\
\vAut_{k} ^{\Str}  K &= (\uAut_k ^{\Str} K\rtimes\gexp{G} K) \rtimes \Ext\Aut_k   K\\
&\simeq (\uAut_k ^{\Str} K\rtimes\Hom(G,k^\times)) \rtimes \oAut G.
\end{align*}\qed
\end{theorem}
\begin{proof}
Proposition~\ref{prop:hofberger-str-linear} and Lemma~\ref{int-str-lin:lemma} yield the equalities.
The isomorphisms are now a consequence of Lemma~\ref{lemma:1Aut-normal} Definition~\ref{def:g-exp} and Remark~\ref{rmk:Psi1-Psi2}.
\end{proof}

\noindent
From now we will focus on the subgroups $ \vAut_{(k)}^{\Str}K $ and $ \vAut_{k}^{\Str}K $.
In Theorem~\ref{thm:hofberger-enhanced-str} we see that all components, except possibly $ \uAut_kK $, only depend  on $ k $ and $ G $.
The next section is devoted to the study of the remaining component: $ \uAut^{\Str}_{k} K $. Recall that, by Lemma~\ref{lemma:1-Aut-k=1-Aut(k)}, we have $ \uAut_{(k)} K = \uAut_k K $.

\subsection[Description of the groups of strongly additive internal and 1-automorphisms]{Description of $ \boldsymbol{\Int\Aut^{\Str}K} $ and $ \boldsymbol{\uAut^{\Str}K} $}\label{subsec:general}

Schilling~\cite{schilling44} describes the group $ \vAut_k\K $, for $ \K =k\pow{\Z} $, in terms of $ U_\K $, the group of units of the valuation ring of $ \K $.
Drawing inspiration from his work,
we aim at an {explicit} description of the groups $ \vAut_{(k)}^{\Str}K $ and $ \vAut_{k}^{\Str}K $, for an arbitrary Hahn field $ K $, in terms of the fundamental objects connected to $  K $. Let $ U=U_K $.
We will further describe the group $ \Int\Aut_k^{\Str} K $ in terms of (a subgroup of) the group $ \Hom(G,U) $. Then we will deduce a description of $ \uAut_{(k)}^{\Str}K $ in terms of a subgroup of $ \Hom(G,1+I_K) $.
In Section~\ref{subsect-laurent-series} we will retrieve Schilling's result as a special case of ours.

Let $ K $ be a Hahn field and $ \sigma\in\Int\Aut_{k} K $.
Recall that $ \sigma $ satisfies the following conditions: for all $ a\in  K $ we have
	 $ v(\sigma(a)) = v(a) $ and	   $ \sigma|_k = \id_k. $
These properties of $ \sigma $ imply that $ \sigma(t^g) = u_\sigma(g)t^g$ for some $ u_\sigma(g) \in U $ depending on $ g $. For all $ \sigma\in\Int\Aut K $ define
$ u_\sigma\colon G\to U $  by $ g\mapsto t^{-g}\sigma(t^g) $.

\begin{lemma}
Let $ K $ be a Hahn field.
	For all $ \sigma \in \Int\Aut K $ the map $ u_\sigma $ is a group homomorphism.
\end{lemma}
\begin{proof}
	Let $ \sigma \in \Int\Aut  K $ and $ g,h\in G $. Then we have
	\begin{align*}
	u_\sigma(g+h) &= t^{-(g+h)}\sigma(t^{g+h})\\
	&= t^{-g}t^{-h}\sigma(t^g)\sigma(t^h)\\
	&= t^{-g}\sigma(t^g)t^{-h}\sigma(t^h)\\
	&=u_\sigma(g)u_\sigma(h).
	\end{align*}
\end{proof}
\noindent This gives rise to a map $ \cS' $ from $ \Int\Aut K $ to the set $ \Hom(G,U) $ of homomorphisms of $G$ into $ U $ defined by $ \cS'
(\sigma)=u_\sigma $.
We are interested in the restriction of $ \cS' $ to $ \Int\Aut_{(k)}^{\Str} K = \Int\Aut_{k}^{\Str} K $, which we will denote by $ \cS $:
\begin{equation}\label{eq:S}
\cS\colon\Int\Aut_{k}^{\Str}  K\to \Hom(G,U)\quad
\sigma\mapsto u_\sigma.
\end{equation}

\begin{lemma}
	\label{S-injective-str-lin}
	The map $ \cS $ defined in \eqref{eq:S} is injective.
\end{lemma}
\begin{proof}
	Let $ \sigma,\tau \in\Int\Aut_{k}^{\Str}  K $ be such that $ \cS(\sigma) = u_\sigma = u_\tau = \cS(\tau) $. Then, for all $ g\in G $, we have $ t^{-g}\sigma(t^g) = u_\sigma(g) = u_\tau(g) = t^{-g}\tau(t^g)$ which implies $ \sigma(t^g) = \tau(t^g) $ and this (since $ \sigma $ and $ \tau $ are strongly additive $ k $-automorphisms) implies that
	$ \sigma(a) = \tau(a) $ for all $ a\in  K $, hence $ \sigma = \tau $.
\end{proof}
\noindent
Now we determine the image of $ \cS $.

\begin{defn}\label{def:summable-map}
	An element $ u\in \Hom(G,U) $ is $ K $-\emph{summable} if, for every $ a \in K $, the family $ \{a_gu(g)t^g:g\in \supp a \}$ is $ K $-summable. Let us denote the set of summable elements of $  \Hom(G,U) $ by $ \Hom^+(G,U) $.
\end{defn}

\begin{lemma}\label{lemma:cS-surjective}
	We have $ \im\cS= \Hom^+(G,U) $. Therefore $ \cS $ corestricts to a bijection
	\begin{align}\label{eq:S-corestr}
	\begin{split}
	\cS\colon\Int\Aut^{\Str}_{k}  K&\to \Hom^+(G,U).
	\end{split}
	\end{align}
\end{lemma}
\begin{proof}
	Let $ u\in\im\cS $.
	Then $ u = u_{\sigma} $ for some $ \sigma\in \Int\Aut_{k}^{\Str} K $.
	Now let $a = \sum a_gt^g \in  K$.
	Since $ \sigma\in \Int\Aut_{k}^{\Str} K $ we have
	$ \sigma(a) = \sum a_g\sigma(t^g) = \sum a_gu_\sigma(g)t^g $,
	hence the family $ \{a_gu_{\sigma}(gt^g): g\in \supp a\} $ is $ K $-summable. Therefore $ u\in\Hom^{\Str}(G,U) $ and so $ \im\cS \subseteq \Hom^{\Str}(G,U) $.
	
	Conversely, let $ u\in  \Hom^{\Str}(G,U) $ and define $ \sigma $ by $ \sigma\brackets{\sum a_gt^g} = \sum a_gu(g)t^g  $ for all $ \sum a_gt^g \in K $. Since $ u $ is $ K $-summable, $ \sigma\in\vAut_k^{\Str}K  $ is well defined.
	Now we show $ \sigma\in\Int\Aut_{k}K $. For $ g\in G $ let $ u(g) = u_0+\epsilon(g) $ with $u_0\in k^\times,\ \epsilon(g)\in I_K $, let $ a = \sum a_gt^g\in K $ and set $ v(a)=h $. Note that $ v(a_gu(g)t^g)=g $ for all $ g\in\supp a $. 
	Moreover, $ \bar{\sigma}(a_0+I_K) = \sigma(a_0)_0 +I_K = a_0u(0)+I_K = a_0+I_K $
	so $ \bar\sigma=\id $ ($ \bar\sigma $ was defined in Remark~\ref{rmk:val-pres-induces-ord}).
	So $ \sigma\in\Int\Aut_k^{\Str}K $ and, by definition of $ \sigma $ we have  $ \sigma(t^g) = u(g)t^g $ thus $ u = u_\sigma \in\im\cS$.
	Hence $ \Hom^{\Str}(G,1+I_K) \subseteq \im\cS $, which completes the proof.
\end{proof}

\begin{defn}\label{def:bullet}
	We define an operation\footnote{This operation corresponds to the crossed representation defined by Schilling for $ \Aut_k\bbL $ where $\bbL= k\pow{\Z} $. See Section~\ref{subsect-laurent-series} for more details.}
	\func{\molt}{\Hom^+(G,U)\times\Hom^+(G,U)}{\Hom^+(G,U)}{(u_\tau,u_\sigma)}{[u_\tau\molt u_\sigma:g\mapsto\tau(u_\sigma(g))u_\tau(g)].}
\end{defn}

\begin{prop}
	\label{prop:int-iso-summ}
	The map $ \cS\colon\Int\Aut^{\Str}_{k} K \ra  \Hom^{\Str}(G,U) $ defined in \eqref{eq:S-corestr} is a group isomorphism, if we equip $ \Hom^{\Str}(G,U) $ with the new operation $ \molt $:
	\begin{equation}\label{eq:cS-isom}
	\cS\colon(\Int\Aut_{k}^{\Str}K,\circ) \overset{\sim}{\lra}  (\Hom^+(G,U),\molt).
	\end{equation}
\end{prop}

\begin{proof}
	By Lemmas~\ref{S-injective-str-lin} and \ref{lemma:cS-surjective} the map \eqref{eq:cS-isom} is bijective.
	It remains to show that it is a group homomorphism.
	Let $ \sigma,\tau \in \Int\Aut^{\Str}_{k} K$ and let $ g\in G $. Then we have
	\begin{align*}
	u_{\tau \sigma}(g) &= t^{-g}(\tau \sigma)(t^g)
	=t^{-g}\tau(\sigma(t^g))
	= t^{-g}\tau(u_\sigma(g)t^g) = t^{-g}\tau(u_\sigma(g))\tau(t^g)\\
	&= t^{-g}\tau(u_\sigma(g))t^gu_\tau(g)
	= \tau(u_\sigma(g))u_\tau(g)
	=(u_\tau \molt u_\sigma)(g).
	\end{align*}
\end{proof}


\begin{cor}\label{cor:X-decomposition}
	Restricting $ \cS $ we get
	\begin{align}
	(\gexp{G}K,\circ) &\simeq \brackets{\Hom(G,k^\times, ),\times}=\brackets{\Hom(G,k^\times), \cdot}\label{eq:S-on-Gexp}\\
	(\uAut_{k}^{\Str}K,\circ) &\simeq \brackets{\Hom^+(G,1+I_K),\times}\label{eq:S-on-1Aut}
	\end{align}
	and thus
	\begin{equation}\label{eq:isom-on-X}
	(\Hom^+(G,U),\molt) \simeq 
	(\Hom^+(G,1+I_K),\molt) \rtimes (\Hom(G,k^\times),\cdot).
	\end{equation}
	
\end{cor}
\begin{proof}
	
	Let $ \rho=\rho_x\in\gexp{G}K $. Then $ \cS(\rho)=u_\rho=x\in \Hom(G,k^\times) $, so $ \cS|_{\gexp{G}K} =  P\inv $, where $  P $ is the map given in Proposition~\ref{prop:embed-G-Exp}. We therefore have an isomorphism
	\begin{equation}
	\cS\colon \gexp{G}K\overset{\sim}{\longrightarrow}  (\Hom(G,k^\times),\molt).
	\end{equation}
	To prove \eqref{eq:S-on-Gexp} we notice that, for $ x,y\in \Hom(G,k^\times) $, corresponding to $ \sigma_x,\sigma_y\in\Int\Aut_{k}^{\Str}K $ we have 
	
	\begin{align*}
	(x\molt y) (g) = x(g)\cdot\sigma_x(y(g)) = x(g)y(g)
	\end{align*}
	because $ y(g)\in k^\times $ and $ \sigma_x\in\gexp{G}K\leq \Aut_kK $. So $ (\Hom(G,k^\times),\molt) = (\Hom(G,k^\times),\cdot) $, and \eqref{eq:S-on-Gexp} follows.
	
	Similarly, if $ \tau\in\uAut_{k}^{\Str}K $ then $ u_\tau\in\Hom(G,1+I_K) $. So $ \Hom^+(G,1+I_K):= (\Hom^+(G,U))\cap  \Hom(G,1+I_K) $. Then we have
	\begin{equation}
	\cS\colon\uAut_{k}^{\Str}K\overset{\sim}{\longrightarrow} \Hom^+(G,1+I_K)
	\end{equation}
	which proves \eqref{eq:S-on-1Aut}. Equation \eqref{eq:isom-on-X} now follows from Proposition~\ref{prop:int-iso-summ} applying \eqref{eq:Int_k^{>0}-semidir}, \eqref{eq:S-on-Gexp} and \eqref{eq:S-on-1Aut}.
\end{proof}

\noindent
Combining Theorem~\ref{thm:hofberger-enhanced-str} and Proposition~\ref{prop:int-iso-summ} we find
\begin{theorem}\label{thm:main}
	Let $ K $ be a Hahn field with the first and canonical second lifting property.
	Then
	\begin{align*}
		\vAut_{(k)}^{\Str}K &\simeq \left(\Hom^+(G,1+I_K)\rtimes \Hom(G,k^\times)\right)\rtimes(\Aut k\times \oAut G);\\
		\vAut_{k}^{\Str}K &\simeq \left(\Hom^+(G,1+I_K)\rtimes \Hom(G,k^\times)\right)\rtimes \oAut G.
	\end{align*}\qed
\end{theorem}

Theorem~\ref{thm:main} thus provides a decomposition of $ \vAut_{(k)}^{\Str}K $ and $ \vAut_{k}^{\Str}K $ purely in terms of the valuation invariants of $ K $.
In the next section we are going to apply the results obtained so far under some further assumptions on the group $ G $ and the field $ k $. This will allow to retrieve results of Schilling~\cite{schilling44} on the field of Laurent series and of Deschamps~\cite{deschamps:puiseux} on the field of Puiseux series.

\section{Explicit examples in special cases}
\label{sec:examples}

\subsection{Finitely generated exponent group}\label{fin-gen-G:subsec}\ \\

Let $ k $ be an arbitrary field and let $ G $ be a totally ordered, finitely generated abelian group. Without loss of generality, we can assume $ G = \Z^n = \prod_{i=1}^n \Z $, for some $ n\in\N $. 
Then $ \Hom(G,k^\times)\simeq (k^\times)^n $.
Let $K\subseteq k\pow{G}$ be a Hahn field satisfying the first and canonical second lifting property
\footnote{This applies, in particular, to $ K = k(\Z^n) $. Then $ \vAut_kK $ is an interesting subgroup of $ \Aut_k K $, which is the Cremona group $ \Cr_n(k) $ (for more on the Cremona group, see \cite{deserti}).}.

Theorem~\ref{hofberger-enhanced} thus yields
\begin{theorem}\label{general-Z^n:thm}
	Let $ G= \Z^n $. Let $ k $ be a field and $ K\subseteq k\pow{G} $ a Hahn field with the first and canonical second lifting property. Then we have
	\begin{align*}
		\vAut K &\simeq (\uAut K \rtimes(k^\times)^n)\rtimes(\Aut k \times \oAut G)\\
		\vAut_k K &\simeq (\uAut_k K \rtimes(k^\times)^n)\rtimes\oAut G.
	\end{align*}
	If, moreover, $ K $ satisfies the canonical first lifting property, Proposition~\ref{hofberger-enhanced-k-stable} yields
	\begin{equation*}
		\vAut_{(k)} K \simeq (\uAut_k K \rtimes(k^\times)^n)\rtimes(\Aut k \times \oAut G).
	\end{equation*}\qed
\end{theorem}

Now we will provide a description of $ \vAut_{(k)}^{\Str}K $ and $ \vAut_k^{\Str}K $.
For $ G= \Z^n $ we have $ \Hom(G,1+I_K)\simeq (1+I_K)^n $.
More precisely, this isomorphism is given as follows. Let $ g_1,\ldots,g_n $ be generators of $ G $, let $ \bar u\in \Hom(G,1+I_K) $ and let $ u_i:= \bar u(g_i) \in 1+I_K $, for $ i=1,\ldots,n $. Then 
$$  \xi\colon\Hom(G,1+I_K)\to (1+I_K)^n ,\ \bar u\mapsto(u_1,\ldots,u_n)  $$ 
is a group isomorphism.
Under $ \xi $, a summable automorphism $ \bar u\in \Hom^{\Str}(G,1+I_K) $ (Definition~\ref{def:summable-map}) corresponds to a tuple $\xi(\bar u)= (u_1,\ldots,u_n) $ such that, for all $ a\in K $ the family 
\[
\setbr{a_g \brackets{\sum r_iu_i} t^g: r_i\in\Z,\ \sum r_iu_i=g,\ g\in\supp a } \]
is $ K $-summable. Let us denote by $ (1+I_K)^{n+}:=\xi\brackets{ \Hom^{\Str}(G,1+I_K)}$. On $ \Hom^{\Str}(G,1+I_K) $ we defined the operation $ \times $ (Definition~\ref{def:bullet}).
We can define an operation on $ (1+I_K)^{n+} $, also denoted by $ \times $, by setting
$ \mathbf{u}_1\times \mathbf{u}_2 := \xi\brackets{\xi\inv(\mathbf{u}_1)\times\xi\inv(\mathbf{u}_2)} $,
for all $ \mathbf{u}_1,\mathbf{u}_1\in (1+I_K)^{n+} $.
We thus obtain

\begin{lemma}\label{Zn-1Aut:lemma}
	$ \Hom^{\Str}(G,1+I_K)\simeq ((1+I_K)^{n+},\times) $. \qed
\end{lemma}

Now assume that $ G=\Z^n $ is equipped with the lexicographic order $ <_{\lex} $
\footnote{For $ g=(g_1,\ldots,g_n)\in G $ we set $ g>_{\lex}0 $ if and only if $ g\neq 0 $ and for the smallest index $ i $ such that $ g_i\neq 0 $ we have $ g_i>0 $.}.
We can explicitly describe $ \oAut G.$ 
Let $ \mathrm{UUT}_n(\Z) $ be the multiplicative group of upper uni-triangular $ n\times n $-matrices with integer coefficients.

\begin{lemma}[{\cite[Lemma 1]{conrad}}\footnote{\cite{conrad} uses upper triangular matrices because he takes the anti-lexicographic ordering on $ G $.}]
	\label{Zn-lex:lemma}
	Let $ G = (\Z^n,<_{\lex}) $. Then $ \oAut G \simeq \mathrm{UUT}_n(\Z) $.\qed
\end{lemma}
%

Now Lemmas~\ref{Zn-1Aut:lemma} and~\ref{Zn-lex:lemma} applied to Theorem~\ref{general-Z^n:thm} provide the following refinement of 
Theorem~\ref{thm:main}.
\begin{theorem}\label{srt-add-Z^n:thm}
	Let $ G= (\Z^n,<_{\lex}) $. Let $ k $ be a field and $ K\subseteq k\pow{G} $ a Hahn field with the first and canonical second lifting property. Then we have
	\begin{align*}
		\vAut_{(k)}^{\Str} K &\simeq (((1+I_K)^{n+},\times) \rtimes(k^\times)^n)\rtimes(\Aut k \times \mathrm{UUT}_n(\Z))\\
		\vAut_k^{\Str} K &\simeq (((1+I_K)^{n+},\times) \rtimes(k^\times)^n)\rtimes\mathrm{UUT}_n(\Z).
	\end{align*}\qed
\end{theorem}

In the next two sections we investigate in more detail the case $ G=\Z $ and provide a more explicit description of the automorphism groups of the field $ \bbL = k\pow{Z} $ of Laurent series and of the function field $ k(\Z) $.

\subsection{Laurent series}
\label{subsect-laurent-series}
Let $k$ be a field and let $\bbL:= k\pow{\Z}$ be the field of formal Laurent series with coefficients in $k$.
This is a maximal Hahn field, thus it has the canonical first and second lifting properties.
On this field the valuation $v$ has residue field $k$ and value group $\Z$.
In \cite{schilling44} Schilling studies the group $ \vAut_k \bbL $ of $ k $-automorphisms of $ \bbL $.
In this section we prove Theorem~\ref{thm:laurent-main}, which is both a generalisation and a refinement of Schilling's result. We also provide a refinement in order to describe the group $\oAut \bbL $, in the case of $ k $ an ordered field (Corollary~\ref{prop:schilling-order}).

We recall that, by Remark~\ref{not:general-KvG-2}, the group of units is $ U:=U_\bbL\simeq (1+I_\bbL)\times k^\times $.

%

\begin{lemma}\label{Hom(Z,U)-is-summable:lemma}
	We have	$ \Hom(\Z,U) = \Hom^{\Str}(\Z,U) $.
\end{lemma}
\begin{proof}
	Let $ \bar u\in\Hom(\Z,U) $, $ u=\bar u (1) $ and $ a\in\bbL $. By Neumann's Lemma~\cite[p.~57]{priess-crampe} the family $ \setbr{a_nu^nt^n:n\in\supp a} $ is $ \bbL $-summable. So $ u\in\Hom^{\Str}(\Z,U) $.
\end{proof}

By Lemma~\ref{Hom(Z,U)-is-summable:lemma} we can use the group structure $ (\Hom(\Z,U),\times) $ described in Definition~\ref{def:bullet} to induce an alternative group structure on $ U $. Call $ \theta\colon\Hom(\Z,U)\to U $ the isomorphism given by $ \theta(\bar u) = u:=\bar u(1) $. Set, for all $ u_1,u_2\in U $
\begin{equation}\label{eq:times-s}
	u_1\times_{s}u_2 = \theta(\theta\inv(\bar u_1)\times\theta\inv(\bar u_2)).
\end{equation}

\begin{lemma}\label{Hom(Z,U)isoU}
	We have
	\begin{align}
		\brackets{\Hom(\Z,U),\times}&\simeq (U,\times_s)\label{Hom(Z,U)isoU-1}\\
		\brackets{\Hom(\Z,k^\times),\times}&\simeq\brackets{\Hom(\Z,k^\times), \cdot}\simeq (k^\times,\cdot)\label{Hom(Z,U)isoU-2}\\
		\brackets{\Hom(\Z,1+I_\bbL),\times}&\simeq (1+I_\bbL,\times_{s})\label{Hom(Z,U)isoU-3}
	\end{align}
	and thus
	\begin{equation}\label{Hom(Z,U)isoU-4}
		\brackets{\Hom(\Z,U),\times}\simeq (1+I_\bbL,\times_{s}) \rtimes (k^\times,\cdot).
	\end{equation}
\end{lemma}
\begin{proof}
	Equation~\eqref{Hom(Z,U)isoU-1} follows immediately from \eqref{eq:times-s}. Equations~\eqref{Hom(Z,U)isoU-2}, \eqref{Hom(Z,U)isoU-3} and~\eqref{Hom(Z,U)isoU-4} are now  special cases of Corollary~\ref{cor:X-decomposition}.
\end{proof}

Next we show that all automorphisms of $ \bbL $ are strongly additive.
\begin{lemma}
	\label{laurent-str-lin}
	We have
	$ \vAut\bbL = \vAut^{\Str}\bbL $.
\end{lemma}

\begin{proof}
	Let $ \sigma \in \vAut\bbL$. By Lemma~\ref{Zn-lex:lemma} with $ n=1 $ it follows that $ \oAut\Z $ is trivial so for all $ a\in \bbL $ we have $ v(a)=v(\sigma(a)) $. 
	Now let $ a=\sum_{i=m}^\infty a_it^i\in\bbL $ with $ m=v(a) $.
	Then the family $ \{ \sigma(a_it^i):i\in\supp a \} $ is summable and, for all $ n\in \Z $ we have
	{\footnotesize \begin{align*}
			v\brackets{\sigma(a)-\sum\sigma\brackets{a_it^i}}
			&=v\brackets{\sigma\brackets{\sum_{i=m}^\infty a_it^i} - \sum_{i=m}^\infty \sigma\brackets{a_it^i}}\\
			&= v\brackets{\sigma\brackets{\sum_{i=m}^n a_it^i}+
				\sigma\brackets{\sum_{i>n}a_it^i} 
				- \sum_{i=m}^n \sigma\brackets{a_it^i}
				+\sum_{>n}\sigma\brackets{a_it^i}}\\
			&= v\brackets{\sum_{i=m}^n \sigma\brackets{a_it^i}+
				\sigma\brackets{\sum_{i>n}a_it^i} 
				- \sum_{i=m}^n \sigma\brackets{a_it^i}
				+\sum_{i>n}\sigma\brackets{a_it^i}}\\
			&= v\brackets{\sigma\brackets{\sum_{i>n}a_it^i}
				+\sum_{i>n}\sigma\brackets{a_it^i}}>n
	\end{align*}}
	hence $ v\brackets{\sigma(a)-\sum\sigma\brackets{a_it^i}}=\infty $ which implies $ \sigma(a)=\sum\sigma\brackets{a_it^i} $.
\end{proof}

The following theorem is now a consequence of Theorem~\ref{srt-add-Z^n:thm} and Lemmas~\ref{Hom(Z,U)isoU} and~\ref{laurent-str-lin}.

\begin{theorem}\label{thm:laurent-main}
	We have
	\[
	\vAut_{(k)}\bbL \simeq ((1+I_\bbL,\times_{s})\rtimes (k^\times,\cdot))\rtimes\Aut k.
	\]	\qed
\end{theorem}


\begin{rmk}\label{laurent-formulas:rmk}
	We can now show explicitly how an automorphism $ \sigma\in\vAut_{(k)}K $ acts. Let $ a = \sum_{i\geq m}a_it^i\in\bbL $. We know that $ \sigma $ is strongly additive, so $ \sigma(a) = \sum \sigma(a_i)\sigma(t)^i $. For all $ i\in\supp a $ we have $ \sigma(a_i)\in k $. Moreover, because $ v(t)=v(\sigma(t))=1 $ we have $ u_\sigma:=t\inv\sigma(t)\in U $. Then $ \sigma $ is uniquely determined by $ u_\sigma $ and $ \sigma|_k $:
	\[
	\sigma\brackets{\sum_{i\geq m}a_it^i} = \sum_{i\geq m}\sigma|_k(a_i)(u_\sigma t)^i.
	\]
	Conversely, to every unit $ u\in U $ and every $ \tau\in\Aut k $ we have the corresponding $ \sigma_{u,\tau}\in\vAut_{(k)}\bbL $ defined by
	\[
	\sigma_{u,\tau}\brackets{\sum_{i\geq m}a_it^i} = \sum_{i\geq m}\tau(a_i)(u t)^i.
	\]\qed
\end{rmk}

\begin{cor}\label{schilling:cor}
	We have
	$ \vAut_{k}\bbL \simeq (U,\times_{s})\simeq (1+I_\bbL,\times_{s})\rtimes (k^\times,\cdot) $.	\qed
\end{cor}

With Corollary~\ref{schilling:cor} we retrieve Schilling's result \cite[Theorem 1]{schilling44}.
To conclude this subsection we sharpen Theorem~\ref{thm:laurent-main} in the case where $k$ is an ordered field, to characterise the group $ \oAut_k\bbL $ of order preserving $ k $-automorphisms of $ \bbL $ (Definition~\ref{def:k,k-stab,v-,o-Aut}).
\begin{cor}\label{prop:schilling-order}
	The $ k $-automorphisms preserving the lexicographic order on $\bbL$ are exactly those corresponding to positive units: $ \oAut_k\bbL \simeq (U^{>0},\times_{s}) $.
	More precisely, we have
	\begin{equation}\label{eq:laurent-semi-direct-ordered}
		\oAut_{(k)}\bbL\simeq \big(\left(1+I_\bbL,\times_{s}\right) \rtimes k^{>0}\big) \rtimes\oAut k.
	\end{equation}
\end{cor}
\begin{proof}
	Let $a = \sum_{i=m}^\infty a_it^i$ with $v(a)=m \in \Z$ and $u=\sum_{i=0}^\infty u_it^i$ a unit ($u_0\neq 0$). We can write $a = t^m\sum_{i=0}^\infty b_it^i$ with $b_i = a_{i-m}$. Let us assume $a>0$, that is $b_0>0$. 
	Let us write $ \sigma_u:=\sigma_{u,\id} $ (see Remark~\ref{laurent-formulas:rmk}).
	Then
	\begin{align*}
		\sigma_u(a) &= \sigma_u\left(t^m\sum_{i=0}^\infty b_it^i\right)
		= \sigma_u(t^m)\sigma_u\left(\sum_{i=0}^\infty b_it^i\right)
		= (tu)^m\left(b_0 + \sigma_u\left(\sum_{i=1}^\infty b_it^i\right)\right)\\
		&= (tu)^m\left(b_0 + [\text{higher order terms}]\right)\\
		&= t^mu_0^mb_0 + [\text{higher order terms]}
	\end{align*}
	hence $\sigma(a)_m = u_0^ma_m>0 $ if and only if $ u_0 > 0 $ or $ m $ is even.
	Thus we have $ \oAut_k\bbL \simeq (U^{>0},\times_{s}) = \left(1+I_\bbL,\times_{s}\right) \rtimes k^{>0} $ (the 1-units are all positive). Now \eqref{eq:laurent-semi-direct-ordered} follows immediately.
\end{proof}


\subsection{The Cremona group in dimension one}\label{cremona-1:es}
Let $ k $ be an arbitrary field.
Consider the Hahn field $ k(\Z)\subseteq \bbL $.
Theorem~\ref{hofberger-enhanced} applies to this field, so we have
\begin{equation}
	\vAut_k k(\Z) = \Int\Aut_k k(\Z) \simeq \uAut k(\Z)\rtimes k^\times.
\end{equation}
\noindent
Note that $ \vAut_kk(\Z) $ is a subgroup of $ \Aut_k k(\Z) $, which is the Cremona group $ \Cr_1(k) $.
It is well known that $ \Cr_1(k) \simeq \PGL_2(k) $ (see, for example, \cite[\S~1.2]{cantat}).
Indeed, $ \sigma\in\Cr_1(k) $ is completely determined by the invertible matrix $ \begin{pmatrix}
	a&b\\c&d
\end{pmatrix} \in k^{2\times 2} $ such that
$ 	\sigma(t)={(at + b)}/{(ct + d)}$. 
We characterise $ \vAut_k k(\Z) $ as a subgroup of $ \Cr_1(k) $ as follows:
\begin{align}\label{eq:moebius-val-equiv}
	\begin{split}
		\vAut_k k(\Z)&= \left\lbrace\sigma\in\Cr_1(k):\sigma(t)=\frac{at}{ct+d},\ a,c,d\in k \text{ with } ad\neq 0\right\rbrace\\
		&= \left\lbrace\sigma\in\Cr_1(k):\sigma(t)=\frac{at}{ct+d},\ \text{ with } v\brackets{\frac{a}{ct+d}}=0\ \right\rbrace.
	\end{split}
\end{align}

Indeed if $ \sigma\in\Int\Aut_k k(\Z) $ then $ 1 = v(t) = v(\sigma(t))= v\brackets{\frac{at + b}{ct + d}} = v(at+b)-v(ct+d) $. This implies 
$ v(at+b)=1 $ and therefore $ a\neq 0 $ and
$ b=0 $.
Conversely,  let $ a,c,d\in k $ with $ ad\neq 0 $. Then
$ u:=\frac{a}{ct+d} $ is a unit in the valuation ring of $ \bbL $, because $ v(u)=0 $.
Therefore, by Lemma~\ref{laurent-str-lin},
$ t\mapsto ut =\frac{at}{ct+d} $ determines a $\sigma_u\in\vAut_k \bbL $.
Thus the restriction $ \sigma_u|_{k(\Z)}\in\vAut_kk(\Z) $, as required.

\noindent
Notice that what we just showed implies, in particular, that every $ \sigma\in\vAut_k k(\Z) $ extends to an automorphism in $ \vAut_k\bbL $. Moreover, since the group of lower triangular matrices is not normal inside $ \PGL_2(k) $, it follows that $ \vAut_k k(\Z) $ is not a normal subgroup of $ \Cr_1(k) $.

\noindent
We also characterise $ \uAut_k k(\Z) $ as a subgroup of $ \Cr_1(k) $ as follows:
\begin{align}\label{eq:moebius-val-equiv}
	\begin{split}
		\uAut_k k(\Z)&= \left\lbrace\sigma\in\Cr_1(k) : \sigma(t)=\frac{at}{ct+a},\quad a,c\in k\quad \text{ with } a\neq 0\ \quad \right\rbrace\\
		&=\left\lbrace\sigma\in\Cr_1(k):\sigma(t)=\frac{at}{ct+a},\ \text{ with } v\brackets{\frac{a}{ct+a}-1}>0\right\rbrace.
	\end{split}
\end{align}
\noindent
Indeed, $ \sigma\in\vAut_k k(\Z) $ (and its extension to $ \bbL $) is defined by $ t\mapsto ut $ with $ u=\frac{a}{ct+d} $ and $ ad\neq 0 $. By Corollary~\ref{schilling:cor} we know that $ \sigma\in\uAut_k \bbL $ if and only if $ u\in 1+I_\bbL $ which is indeed equivalent to the condition $ a=d $.

\subsection{Divisible and finite dimensional exponent group}
\label{divisible:subsec}

In this subsection we consider the special case of a Hahn field $ K\subseteq k\pow{G} $ where $ G $ is uniquely divisible and finite dimensional (as a $ \Q $-vector space).
\begin{enumerate}[$ \circ $]
	\item If $ G $ is ordered lexicographically, we know precisely what $ \oAut G $ is.
	\item If $ k $ is real closed, we get an explicit description of the group $ \Hom(G,k^\times) $.
	\item If $ K $ is henselian of characteristic 0, we can explicitly describe $ \Hom(G,1+I_K) $. 
\end{enumerate}

\begin{defn}
	A \emph{divisible group} is an abelian group $ H $ such that, for every $ h\in H $ and every $ n\in\Z $ there exists $ h'\in H $ such that $ h=nh' $. If the choice of $ h' $ is unique then $ H $ is called \emph{uniquely divisible} (i.e., $ H $ is uniquely divisible if and only if it is divisible and torsion free).
\end{defn} 
\noindent
A uniquely divisible group is canonically a vector space over $ \Q $.
Throughout this subsection, let $ G $ be a divisible, totally ordered, abelian group which is finite dimensional as a vector space over $ \Q $.
In particular $ G $ is uniquely divisible. 
Set $ d=\dim_\Q G $. Without loss of generality we can assume $ G = \Q^d $.

\begin{rmk}
	If $ H $ is a uniquely divisible group then every group homomorphism $ \theta\in\Hom(G,H) $ is $ \Q $-linear.
	It follows that $ \Hom(G,H)\simeq H^d $.\qed
\end{rmk}

\subsubsection{Lexicographically ordered exponent group}
Assume that $ G = \Q^d $ is equipped with the lexicographic ordering.  Let $ \mathrm{UPT}_d(\Q) $ be the multiplicative group of upper triangular $ d\times d $-matrices over $ \Q $ with positive diagonal entries:
\[
\mathrm{UPT}_d(\Q) = \setbr{\brackets{q_{i,j}}_{i,j=1}^d:
	q_{ij}\in \Q \text{ and } \begin{cases}
		q_{ij}=0 \text{ for } i>j\\
		q_{ij}>0 \text{ for } i=j
	\end{cases} \text{ for }i,j=1,\ldots,d}.
\]
\noindent
Since $ \oAut \Q \simeq (\Q^{>0},\cdot) $ we have
\begin{lemma}[{\cite[Lemma~1]{conrad}}]
	\label{Qd-lex:lemma}
	We have
	$ \oAut G \simeq \mathrm{UPT}_d(\Q) $.\qed
\end{lemma}

\subsubsection{Real closed coefficient field} Let $ k $ be a real closed field and let $ k^{>0} $ be the multiplicative subgroup of positive elements of $ k $.

\begin{lemma}
	We have $ \Hom((\Q,+),(k^\times,\cdot))=\Hom((\Q,+),(k^{>0},\cdot)). $
\end{lemma}

\begin{proof}
	Let $ \theta\in\Hom((\Q,+), (k^\times,\cdot)) $.
	We need to show that $ \theta(\Q)\subseteq k^{>0} $.
	Let $ q\in \Q $.
	Then $ q = 2\frac{q}{2} $.
	Therefore 
	\[
	\theta(q) = \theta\brackets{2\frac{q}{2}} = \theta\brackets{\frac{q}{2}}^2> 0.
	\]
\end{proof}
\begin{lemma}
	The group $ (k^{>0},\cdot) $ is uniquely divisible.
\end{lemma}

\begin{proof}
	Let $ x\in k^{>0} $ and $ n\in \N $ with $ n>0 $.
	Since $ k $ is real closed, there exists $ y\in k^{>0} $ such that $ y^n=x $, so $ k^{>0} $ is divisible.
	Moreover, assume that $ z\in k^{>0} $ is such that $ z^n=x=y^n $ and that we have $ y\neq z $. We may assume $ y>z $ (the case $ y<z $ is identical).
	Because $ y $ and $ z $ are both positive, it follows that $ y^n>z^n $. A contradiction.
	So $ y $ is unique and the proof is complete.
\end{proof}

\begin{cor}\label{k-rcf:lemma}
	We have $ \Hom(G,(k^\times,\cdot))\simeq (k^{>0},\cdot)^d $. In particular, we also have $ \Hom((\Q,+),(k^\times,\cdot))\simeq (k^{>0},\cdot) $.\qed
\end{cor}

\subsubsection{Henselian Hahn field}


Let $ G $ be an arbitrary ordered abelian group.
Let $ k $ be an arbitrary field with $ \Char k =0 $ and let $ K \subseteq k\pow{G} $ be a henselian Hahn field (Definition~\ref{henselian:def}).
Denote by $ \mu(K)=\{ x\in K^\times : x^n=1 \text{ for some } n\in\N\setminus\setbr{0} \} $ the multiplicative group of roots of unity in $ K $.
Note that $ \mu(K)\subseteq U_K $. Indeed, if $ a\in \mu(K) $ there exists $ n\in \N\setminus\{0\} $ such that $ a^n=1 $. Therefore $ 0=v(1)=v(a^n)=nv(a) $ implies $ v(a)=0 $. So $ a\in U_K$.
Moreover $ a^n=1 $ implies $ (a_0)^n=(a^n)_0=1 $.
Thus the map $ \cR\colon(\mu(K),\cdot)\to (\mu(k),\cdot) $, $ a\mapsto a_0 $
is a well defined group homomorphism.

\begin{lemma}
	\label{henselian-roots-unity:lemma}
	The map $ \cR $	is an isomorphism.
\end{lemma}

\begin{proof}
	Let us prove the surjectivity: let $ \alpha\in k $ be such that $ \alpha^n=1 $ for some $ n\in\N\setminus\setbr{0} $.
	Then $ \alpha $ is a simple root of $ X^n-1\in k[X] $.
	Because $ K $ is henselian there exists $ a\in R_K $ such that $ a_0=\alpha $ and $ a^n-1=0 $. Thus $ \cR(a)=\alpha $, which proves surjectivity.
	
	Now let $ u\in\ker\cR $. Then $ u_0=1 $ and there exists $ n\in \N\setminus\setbr{0} $ such that $ u^n=1 $. So $ u\in 1+I_K $, that is, $ u $ is of the form $ u=1+\epsilon $ for some $ \epsilon\in I_K $. Then the binomial expansion gives $ 1=u^n=(1+\epsilon)^n = \sum_{j=0}^n\binom{n}{j}\epsilon^j = 1+\sum_{j=1}^n\binom{n}{j}\epsilon^j $.
	Thus
	\begin{equation}\label{eq:epsion}
		\epsilon\sum_{j=1}^n\binom{n}{j}\epsilon^{j-1}=0.
	\end{equation}
	Because $ v(\epsilon^i)\neq v(\epsilon^j) $ for $ i\neq j $ the strict ultrametric inequality implies that $ v\brackets{\sum_{j=1}^n\binom{n}{j}\epsilon^{j-1}}=v(n)=0 $ and so $ \sum_{j=1}^n\binom{n}{j}\epsilon^{j-1}\neq 0 $.
	Since $ \Char K=0 $ we deduce from \eqref{eq:epsion} that $ \epsilon =0 $ and thus $ u=1 $ as required.
\end{proof}

\begin{prop}
	Let $\Char k =0 $ and let $ K\subseteq k\pow{G} $ be a henselian Hahn field.
	The multiplicative group $ (1+I_K,\cdot) $ is uniquely divisible.
\end{prop}

\begin{proof}
	Let $ a\in 1+I_K $ and let $ n\in\N\setminus\setbr{0} $. We want to show the existence and the uniqueness of a $ b\in 1+I_K $ such that $ b^n=a $. 
	Consider the polynomial $ P = X^n-a\in R_K[X] $.
	Because $ a\in 1+I_K $ we have $ a_0 = 1 $.
	The polynomial $ \bar P = X^n-1 \in k[X] $ has the root $ x=1 $ in $ k $, which is simple because $ \bar P'(1)=n $ and $ \Char k =0 $.
	Since $ (K,v) $ is henselian there exists $ b\in R_K $ such that $ P(b)=0 $ and $ b_0 = a_0 =1 $. So $ b^n=a $ and $ b\in 1+I_K $. This proves the existence.
	To prove uniqueness, let $ c\in 1+I_K $ be such that $ c^n= a $. Then $ b^n=c^n $ and so $ (b/c)^n=1 $. Therefore, $ b/c\in 1+I_K $ is an $ n $-th root of unity and $ (b/c)_0 =1 $. But $ 1 $ is also an $ n $-th root of unity with $ (1)_0 = 1 $. By Lemma~\ref{henselian-roots-unity:lemma} we must have $ b/c=1 $ and therefore $ b=c $.
\end{proof}

Assume now that $ G = \Q^d $.
\begin{cor}\label{henselian-1+I_K-uni-div:cor}
	We have $ \Hom((G,+), (1+I_K,\cdot))\simeq (1+I_K)^d $. In particular, $ \Hom((\Q,+), (1+I_K,\cdot))\simeq 1+I_K $.\qed
\end{cor}
\noindent
A summable automorphism $ u\in \Hom^{\Str}(G,1+I_K) $ (Definition~\ref{def:summable-map}) corresponds to a tuple $ (u_1,\ldots,u_d)\in (1+I_K)^d $ such that, for all $ a\in K $ the family 
\[
\setbr{a_g \brackets{\sum q_iu_i} t^g: q_i\in\Q,\ \sum r_iu_i=g,\ g\in\supp a } \]
is $ K $-summable. Let us denote by $ (1+I_K)^{d+}$ the subgroup of $ (1+I_K)^d $ corresponding to $ \Hom^{\Str}(G,1+I_K) $, equipped with the operation $ \times $ induced by that on $ \Hom^{\Str}(G,1+I_K) $ (same exact procedure as in Subsection~\ref{fin-gen-G:subsec}).
We therefore have
\begin{cor}\label{henselian-1Aut:cor}
	We have
	$ \Hom(G,1+I_K)\simeq ((1+I_K)^{d+},\times) $. \qed
\end{cor}

Combining Lemmas~\ref{Qd-lex:lemma} and~\ref{k-rcf:lemma} and Corollary~\ref{henselian-1+I_K-uni-div:cor} we obtain the following refinement of Theorem~\ref{thm:main}.

\begin{theorem}\label{hensel-char0-str-add:thm}
	Let $ k $ be a real closed field, $ G = (\Q^d,<_{\lex}) $ and $ K\subseteq k\pow{G} $ a henselian Hahn field satisfying the first and canonical second lifting property. Then
	\begin{align*}
		\vAut_{(k)}^{\Str} K &\simeq (((1+I_K)^{d+},\times) \rtimes(k^\times)^d)\rtimes(\Aut k \times \mathrm{UPT}_d(\Q))\\
		\vAut_k^{\Str} K &\simeq (((1+I_K)^{d+},\times) \rtimes(k^\times)^d)\rtimes\mathrm{UPT}_d(\Q).
	\end{align*}\qed
\end{theorem}

In the next section we analyse in further detail a special case for $ G=\Q $, namely the field $ \bbP $ of Puiseux series.

\subsection{Puiseux series}
\label{subsect-puiseux}
Let $ k $ be a real closed field and let $ \bbP $ be the field of Puiseux series in the indeterminate $t$ over $ k $. 
These are power series with coefficients in the field $k$ and exponents in $ \Q $ with the restriction that all the exponents of a given power series have a common denominator. A general Puiseux series has the form:
\begin{equation}
	\label{puiseux-typical}
	a = \sum_{n=m}^\infty a_n t^\frac{n}{n_a}
\end{equation}
where $ m\in \Z $ and $ n_a \in \Z^{>0} $ is a positive integer depending on $a$.
The field $ \bbP $ is a subfield of the Hahn field $ k\pow{\Q} $ (Example~\ref{es:rayner-puiseux}).
It therefore has value group $ (\Q,+,<) $. By Lemma~\ref{rayner-closed-G-Exp} and Corollary~\ref{canonical-puiseux-kappa} the field $ \bbP $ has the canonical first and second lifting property.
%
%
Moreover, like in the case of Laurent series, all the valuation preserving automorphisms of $ \bbP $ are strongly additive:

\begin{prop}\label{prop:puiseux-str-lin}
	We have $ \vAut \bbP = \vAut^{\Str}\bbP $.
\end{prop}

\begin{proof}
	Let $ \sigma\in \vAut_{(k)}\bbP $.
	We showed earlier (Remark~\ref{external-are-strongly-lin}) that external automorphisms are always strongly additive. So let us assume $ \sigma $ to be internal. By Theorem~\ref{hofberger-enhanced-k-stable} we have $ \sigma\in\vAut_k \bbP $.
	Since $ \sigma $ is internal, then for all $ a\in \bbP $ we have $ v(\sigma(a)) = v(a) $.
	In particular, $ v(\sigma(t^\frac{1}{n}))>0 $ for all $ n\in \N $ and by Neumann's lemma~\cite[Lemma~15, p.~57]{priess-crampe}, the family $ \left\lbrace \sigma\left(a_nt^\frac{1}{n_a}\right)^n: n\in \N \right\rbrace $ is summable, hence $ \sum \sigma(a_nt^\frac{1}{n_a})\in \bbP $ is well defined.
	To prove our statement we show that  the value $ v(\sum_{n=m}^\infty  (\sigma (a_nt^\frac{1}{n_a}))^n - \sigma(a)) $ is greater than $ \frac{s}{n_a} $ for all $ s\in\Z $. Indeed:
	{\footnotesize 	\begin{align*}
			&v\left( \sum_{n\geq m} \sigma \left(a_nt^\frac{1}{n_a}\right)^n - \sigma \left(\sum_{n\geq m}a_n t^\frac{n}{n_a}\right) \right) =
			v\left(
			\sum_{n> s} \sigma \left(a_nt^\frac{n}{n_a}\right)
			- 
			\sigma \left(\sum_{n>s}a_n t^\frac{n}{n_a}\right)	
			\right)\geq
			\\
			&\min\left\lbrace 
			v\left(
			\sum_{n> s} \sigma \left(a_nt^\frac{n}{n_a}\right)
			\right), v\left(
			\sigma \left(\sum_{n>s}a_n t^\frac{n}{n_a}\right)	
			\right)
			\right\rbrace > \frac{s}{n_a}.
	\end{align*}}
	Thus $ v(\sum_{n=m}^\infty  (\sigma (a_nt^\frac{1}{n_a}))^n - \sigma(a))=\infty $ and therefore $ \sigma(a) = \sum_{n=m}^\infty  \sigma (a_nt^\frac{n}{n_a}) $.
\end{proof}

The field $ \bbP $ is henselian \cite[Lemma~10.1]{Kuhlmann2000}, so 
Theorem~\ref{hensel-char0-str-add:thm} applies. Combining this with Proposition~\ref{prop:puiseux-str-lin} we get

\begin{theorem}
	Let $ k $ be a real closed field. Then
	\begin{align}
		\vAut_{(k)}\bbP&\simeq \left((1+I_\bbP,\times_{s})\rtimes k^\times\right)\rtimes(\Aut k\times (\Q^{>0},\cdot))\\
		\vAut_{k}\bbP&\simeq \left((1+I_\bbP,\times_{s})\rtimes k^\times\right)\rtimes (\Q^{>0},\cdot)
	\end{align}\qed
\end{theorem}

The case where $ k $ is an algebraically closed field of characteristic 0 was treated by
Deschamps \cite[Th\'eor\`eme 10]{deschamps:puiseux}.
Under this assumption, he proves that $ \uAut_k\bbP $ and $ \gexp{\Q}\bbP $ can be described, respectively, as $ \uAut_k\bbP \simeq \varinjlim (1+I_\bbP) $ and $ \gexp{\Q}k \simeq \varprojlim k^\times $, where the limits are taken over the directed system given by the natural numbers with divisibility.

%
%

\printbibliography
\end{document}